\documentclass[runningheads]{llncs}
\usepackage[T1]{fontenc}
\usepackage[utf8]{inputenc}
\usepackage{graphicx}
\usepackage{xcolor}
\usepackage{amsmath}
\usepackage{amssymb}

\usepackage{amsthm}
\usepackage[hidelinks]{hyperref}
\usepackage{algorithm}
\usepackage{algpseudocode}
\usepackage{booktabs}
\usepackage{subcaption}

\newtheoremstyle{dotstyle}
  {\topsep}   
  {\topsep}   
  {\itshape}  
  {}          
  {\bfseries} 
  {.}         
  {.5em}      
  {}          

\theoremstyle{dotstyle}
\newtheorem{assumption}{Assumption}

\newcommand{\bmu}{\boldsymbol{\mu}}
\newcommand{\bsigma}{\boldsymbol{\sigma}}

\begin{document}
\title{Safe Optimal Control using Log Barrier Constrained iLQR}
%
%
\author{Abhijeet \and
Suman Chakravorty} 
%
\authorrunning{Abhijeet and S. Chakravorty}
%
\institute{Texas A\&M University, College station, TX, USA}
%
\maketitle              
\begin{abstract}
This paper presents a constrained iterative Linear Quadratic Regulator (iLQR) framework for nonlinear optimal control problems with box constraints on both states and control inputs. We incorporate logarithmic barrier functions into the stage cost to enforce box constraints (upper and lower bounds on variables), yielding a smooth interior-point formulation that integrates seamlessly with the standard iLQR backward–forward pass. The Hessian contributions from the log barriers are positive definite, preserving and enhancing the positive definiteness of the quadratic approximations in iLQR and providing an intrinsic regularization effect that improves numerical stability and convergence. Moreover, since the negative logarithm is convex, the addition of log barrier terms preserves convexity if the cost is already convex. We further analyze how the barrier-augmented iLQR naturally adapts feedback gains near constraint boundaries. In particular, at convergence, the feedback terms associated with saturated control channels go to zero, recovering a purely feedforward behavior whenever control is saturated. Numerical examples on constrained nonlinear control problems demonstrate that the proposed method reliably respects box constraints and maintains favorable convergence properties.\footnote{Project website with code and supplementary videos: \href{https://abhinir.github.io/Box-iLQR}{Box-iLQR}.}

\end{abstract}
\section{Introduction}

Optimal control theory provides a powerful mathematical framework for determining control strategies that steer a dynamical system along a desired trajectory while minimizing a specified cost function~\cite{bryson1969}. Its applications are fundamental to numerous fields, including robotics, aerospace engineering, economics, and process control. Solving optimal control problems (OCPs) for general nonlinear systems has driven the development of a wide array of numerical methods. These are broadly categorized as indirect methods, which solve the two-point boundary value problem derived from Pontryagin's Minimum Principle, and direct methods, which transcribe the OCP into a finite-dimensional nonlinear programming (NLP) problem~\cite{betts1998survey,betts2010practical,nocedal2006numerical}.

Within the class of indirect methods, the family of Differential Dynamic Programming (DDP)~\cite{mayne1965ddp,wang2022leg_based_sconmip,bhatia2021trajectory_slipping,ortiz2020viability_ddp,lantoine2012hybrid,lantoine2012hybrid_part2}, and its variant, the Iterative Linear Quadratic Regulator (iLQR)~\cite{ILQR_todorov}, are particularly successful. These methods exploit the sequential structure of the OCP, using a backward pass to construct local quadratic approximations of the value function and a forward pass to simulate the system with an updated control policy. This sequential decomposition makes iLQR highly efficient, as it transforms a large optimization problem over an entire trajectory into a sequence of smaller, more manageable problems at each time step. Recent work has further explored the connection between DDP/iLQR and sequential quadratic programming, providing deeper insights into their convergence properties~\cite{abhi_DDP_ILQR}. In comparison, iLQR has robust convergence properties, even when initialized far from the optimal solution, and is highly effective for complex control problems~\cite{rand2c,parunandi2020d2c,naveed_feedback,Sharma2026Reduced,11107931,naveed_infinite,raman_pod,raman_tnnls}. 

Despite these advantages, applying iLQR to real-world systems presents a significant challenge: handling state and control constraints. This has led to a shift towards using other optimization methods for constrained motion planning~\cite{zucker2013chomp,kalakrishnan2011storc,schulman2013finding}. Attempts at constrained-iLQR include approaches like clamping controls and often show slower convergence~\cite{CL_DDP,mastalli2020fddp}. Other methods also have notable drawbacks: gradient projection techniques~\cite{CL_DDP,Giftthaler2017APA,Xie2017DifferentialDP,emo_ILQG} can yield suboptimal solutions, while Augmented Lagrangian (AL) formulations often require complex parameter tuning and does not guarantee convergence to the optimal~\cite{Jackson2019ALiLQRT,howell2019altro,lin1991differential}. Although Interior Point Methods (IPMs) offer a promising direction~\cite{DDP_IPM,rao1998application,vanroye2023fatropfastconstrained,FATROP}, many existing implementations lack robust convergence guarantees. This paper addresses this critical gap by introducing a constrained iLQR algorithm founded on a barrier-based IPM. Crucially, our method is easy to tune and guarantees convergence under some mild assumptions. Furthermore, it offers significant advantages, including inherent regularization from the barrier terms, which enhances numerical stability. By integrating these features into the iLQR framework, we retain its superior global convergence properties~\cite{abhi_DDP_ILQR} while rigorously and efficiently handling constraints.\\

In this paper, we focus on a particularly important class of constraints: box constraints on both states and controls ($x_{\min} \le x_k \le x_{\max}$, $u_{\min} \le u_k \le u_{\max}$). We propose a novel iLQR algorithm that directly handles these constraints using a logarithmic barrier Interior Point Method. Barrier methods, pioneered in works such as~\cite{fiacco1968nonlinear}, transform a constrained problem into a sequence of unconstrained ones by adding a barrier term to the cost function. This term penalizes solutions that approach the boundary of the feasible set, ensuring that iterates remain strictly feasible. The use of IPMs, originating with Karmarkar's algorithm for linear programming~\cite{karmarkar1984new}, has since become a dominant paradigm for general nonlinear optimization~\cite{wright1992interior,wright1997primal,nocedal2006numerical,waltz2006interior}, with robust implementations widely available~\cite{wachter2006implementation}. By integrating this powerful framework into iLQR, our approach allows the backward pass to account for the constraints and proactively adjust the control policy, resulting in a more robust and efficient optimization process.

The main contributions of this work are:
\begin{itemize}
    \item A novel formulation of the box-constrained trajectory optimization problem within a barrier-based Interior Point framework, tailored for the iLQR algorithm.
    \item A comprehensive algorithm, which we refer to as Box-iLQR, that systematically reduces the barrier parameter to converge to a solution of the original constrained problem.
    \item Analysis of the various advantages of using Box-iLQR (Section~\ref{sec:alg_properties}).
    \item Numerical validation on several benchmark optimal control problems, demonstrating the effectiveness of our proposed method.
\end{itemize}

The remainder of this paper is organized as follows. Section~\ref{sec:background} provides a brief overview of the constrained optimal control problem being addressed in the paper. Section~\ref{sec:method} details our Box-iLQR algorithm, followed by its advantages in Section~\ref{sec:alg_properties}. Section~\ref{sec:results} presents the numerical results, and Section~\ref{sec:conclusion} concludes the paper with a discussion of future work.

\section{Background}\label{sec:background}
We start by posing the box constrained optimal control problem in this section. Followed by which, we pose the unconstrained optimization problem using log barriers.
\subsection{Constrained Optimal Control Problem}
Consider the following optimal control problem:
\begin{subequations} \label{eq:ocp}
    \begin{align}
        \min_{\mathbf{u}(\cdot)} \quad & \mathcal{J} = \phi(\mathbf{x}(t_f)) + \int_{0}^{t_f} c(\mathbf{x}(t),\mathbf{u}(t))\,dt  \\
        \text{s.t.} \quad & \dot{\mathbf{x}}(t) = \mathbf{f}(\mathbf{x}(t),\mathbf{u}(t);t) \quad \forall ~t \in [0,t_f],\label{eq:ocp_dyn} \\
        & \mathbf{x}(t) \in \mathcal{X} := \{\mathbf{x} \in \mathbb{R}^n \mid \mathbf{\underline{x}} \leq \mathbf{x} \leq \mathbf{\bar{x}}\} \label{eq:ocp_x_bounds}, \\
        & \mathbf{u}(t) \in \mathcal{U} := \{\mathbf{u} \in \mathbb{R}^m \mid \mathbf{\underline{u}} \leq \mathbf{u} \leq \mathbf{\bar{u}}\} \label{eq:ocp_u_bounds}, \\
        & \mathbf{x}(0) ~~\text{and}~~ t_{f} ~~\text{are given.}\nonumber
    \end{align}
\end{subequations}
In the above equation, $\mathbf{x}(t) \in \mathbb{R}^n$ is the state and $\mathbf{u}(t) \in \mathbb{R}^m$ is the control input. The function $\mathbf{f}:\mathbb{R}^n \times \mathbb{R}^m \times \mathbb{R} \to \mathbb{R}^n$ describes the system dynamics. The objective functional $\mathcal{J}$ is composed of a terminal cost $\phi:\mathbb{R}^n \to \mathbb{R}$ and a running cost $c:\mathbb{R}^n \times \mathbb{R}^m \to \mathbb{R}$. The state and control are constrained to lie within the compact sets $\mathcal{X}$ and $\mathcal{U}$, respectively, for all $t \in [0, t_f]$. 

The discrete counterpart of the problem mentioned in eq. \eqref{eq:ocp} is:
\begin{subequations} \label{eq:ocp_dis}
    \begin{align}
        \min_{\mathbf{u}_t} \quad & \mathcal{J} = \Phi(\mathbf{x}_T) + \sum_{t=0}^{T-1} C(\mathbf{x}_t,\mathbf{u}_t)\ , \\
        \text{s.t.} \quad & {\mathbf{x}}_{t+1} = \mathbf{F}(\mathbf{x}_t,\mathbf{u}_t;t) \quad \forall~ t=0,1\cdots T-1, \label{eq:ocp_dyn_dis} \\
        & \mathbf{x}_t \in \mathcal{X} := \{\mathbf{x} \in \mathbb{R}^n \mid \mathbf{\underline{x}} \leq \mathbf{x} \leq \mathbf{\bar{x}}\} \label{eq:ocp_x_bounds_dis}, \\
        & \mathbf{u}_t \in \mathcal{U} := \{\mathbf{u} \in \mathbb{R}^m \mid \mathbf{\underline{u}} \leq \mathbf{u} \leq \mathbf{\bar{u}}\}, \label{eq:ocp_u_bounds_dis}\\
        & \mathbf{x}_{0} ~,~ T ~~\text{are given} \nonumber
    \end{align}
\end{subequations}
and $\Phi, C(\cdot),$ and $\mathbf{F}(\cdot, \cdot)$ are the discrete equivalents of $\phi, c$ and $\mathbf{f}$, respectively.
\begin{assumption}[Smooth Dynamics]
\label{ass:smooth_dyn}
    $\mathbf{f}(\cdot,\cdot)$ is a continuously differentiable ($\mathcal{C}^{1}$) function in the interval $[0, t_{f}]$ for all $\mathbf{x} \in \mathcal{X}$ and $\mathbf{u} \in \mathcal{U}$.
\end{assumption}

\begin{lemma}
\label{lemma:dis_con_eq}
    Under assumption \ref{ass:smooth_dyn}, every continuous-time optimal control problem represented in eq. \eqref{eq:ocp} can be represented as a discrete-time optimal control problem represented by eq. \eqref{eq:ocp_dis} using a zero-order control hold for some small $\Delta t_{i}$, $i = 1,\cdots T$ and $\sum_{i} \Delta t_{i} = t_{f}$. 
\end{lemma}

\begin{proof}
Suppose the optimal control problem represented by eq. \eqref{eq:ocp} is given. It follows from assumption \ref{ass:smooth_dyn} that there exists sufficiently small $\Delta t_{i}$ such that $\mathbf{f}(\mathbf{x}(t),\mathbf{u}(t))$ is monotonic over all the intervals, $[0,\Delta t_{1}], [\Delta t_{1}, \Delta t_{1} + \Delta t_{2}], [\Delta t_{1} + \Delta t_{2} , \Delta t_{1} + \Delta t_{2} + \Delta t_{3}],\cdots \Big[\sum_{i=1}^{T-1}\Delta t_{i}, t_{f}\Big]$.
Let us define intervals $I_{1} = [0, \Delta t_{1}], I_{2} = [\Delta t_{1}, \Delta t_{1} + \Delta t_{2}], \cdots I_{T} = \Big[\sum_{i=1}^{T-1}\Delta t_{i}, t_{f}\Big]$. Using these intervals and a set of $T$ control inputs $\mathbf{u}_{t} \in \mathcal{U}$, one can discretize the dynamics in eq. \eqref{eq:ocp_dyn} as
\begin{align}
    \mathbf{x}_{t+1} = \int_{I_{t+1}} \mathbf{f}(\mathbf{x}(t),\mathbf{u}_{t};t)dt \quad \quad t=0,1,\cdots T-1
    \label{eq:discretize_dyn}
\end{align}
where the initial condition on state on each integral is derived from the previous integral and the initial condition on state at $t=0$ is $\mathbf{x}_0 = \mathbf{x}(0) $.
One can see the analogy that $\mathbf{F}(\cdot,\cdot)$ in eq. \eqref{eq:ocp_dyn_dis} is same as $\int_{I_{t}} \mathbf{f}(\mathbf{x}(t),\mathbf{u}_{t};t)dt$ as in eq. \eqref{eq:discretize_dyn}. In addition to this, since $\mathbf{f}(\cdot,\cdot)$ is monotonic, if $\mathbf{x}_{t} \in \mathcal{X}~~ \forall~t$, i.e, if every $\mathbf{x}_{t}$ satisfies the constraint in eq. \eqref{eq:ocp_x_bounds_dis}, then the constraint in eq. \eqref{eq:ocp_x_bounds} is also satisfied by the monotonicity of $f(\cdot,\cdot)$. Furthermore, if every $\mathbf{u}_t$ satisfies eq. \eqref{eq:ocp_u_bounds_dis}, then $\mathbf{u}(t)$ also respects the box constraints in eq. \eqref{eq:ocp_u_bounds} as $\mathbf{u}(t) = \mathbf{u}_t, t \in I_{t}$. In addition to this, one can also write
\begin{align}
    C(\mathbf{x}_{t},\mathbf{u}_{t}) = \int_{I_{t+1}}c\Big(\mathbf{f}(\mathbf{x}(t), \mathbf{u}_t), \mathbf{u}_{t}\Big)dt\quad \text{and} \quad \Phi() = \phi().
    \label{eq:discretize_cost}
\end{align}
So, eq. \eqref{eq:discretize_dyn} and eq. \eqref{eq:discretize_cost} represent a transformation from the continuous-time optimal control problem in eq. \eqref{eq:ocp} to the discrete-time optimal control problem in eq. \eqref{eq:ocp_dis} while respecting constraints.
\end{proof}

\begin{remark}
    This result extends directly to first-order hold on control. The control is defined as a piecewise linear function, which is monotonic over each discretization interval. By the monotonicity of this linear interpolant, if the control values at an interval's endpoints satisfy the box constraints \eqref{eq:ocp_u_bounds}, then all intermediate values of the control are also guaranteed to satisfy these constraints.
\end{remark}

In this work, we solve the discrete-time optimal control problem and take the intervals $\Delta t_{i}$ to be of equal length ($\Delta t_{i} = \Delta t~~\forall ~~i$) when converting the continuous-time problem to the discrete-time problem along with zero-order control hold.

\subsection{The Log Barrier Problem}

The discrete-time optimal control problem in \eqref{eq:ocp_dis} is a large-scale nonlinear program (NLP) amenable to solution by a primal-dual interior point method~\cite{nocedal2006numerical}. This approach reformulates the constrained problem as a sequence of unconstrained problems by incorporating the inequality constraints into the objective function via a logarithmic barrier term and addressing the equality constraints using Lagrange multipliers.

To formalize the treatment of inequality constraints, we define index sets that identify the constrained components. Let $\mathcal{I}_x \subseteq \{1, \dots, n\}$ be the set of indices corresponding to the state variables subject to box constraints. The number of constrained states is given by the cardinality of this set, $|\mathcal{I}_x| = n_1$. Consequently, the number of unconstrained (free) state variables is $n_2 = n - n_1$. Analogously, let $\mathcal{I}_u \subseteq \{1, \dots, m\}$ be the index set for the constrained control inputs, with cardinality $|\mathcal{I}_u| = m_1$. The number of free control inputs is then $m_2 = m - m_1$. This notation allows for any arbitrary subset of state and control components to be constrained.

The box constraints on the state and control variables at each discrete time step $k$ can then be expressed for only the components specified by these index sets:
\begin{align}
    \underline{x}_{i} &\leq x_{t,i} \leq \bar{x}_{i}, \quad \forall i \in \mathcal{I}_x, \quad \forall t \in \{0, \dots, T\} \label{eq:state_box_constraints_idx} \\
    \underline{u}_{j} &\leq u_{t,j} \leq \bar{u}_{j}, \quad \forall j \in \mathcal{I}_u, \quad \forall t \in \{0, \dots, T-1\} \label{eq:control_box_constraints_idx}
\end{align}
where $x_{t,i}$ is the $i$-th component of the state vector at time step $t$, and $\underline{x}_{i}$ and $\bar{x}_{i}$ are its corresponding lower and upper bounds. The notation is analogous for the control $u_{t,j}$.
Next, we can incorporate these inequality constraints into the cost function using a logarithmic barrier method. The constrained optimal control problem in \eqref{eq:ocp_dis} is thereby transformed into an unconstrained problem by augmenting the objective with barrier terms and the dynamic constraints with Lagrange multipliers. This leads to the formulation of the following barrier subproblem, where we seek the stationary point of the augmented cost function $\mathcal{J}_{\boldsymbol{\mu},\boldsymbol{\sigma}}$:
\begin{align} \label{eq:augmented_cost}
    &\mathcal{J}_{\boldsymbol{\mu},\boldsymbol{\sigma}} = \Phi(\mathbf{x}_T) -  \sum_{i \in \mathcal{I}_x} \mu_{i}\left( \log(x_{T,i} - \underline{x}_{i}) + \log(\bar{x}_i - x_{T,i}) \right) \nonumber \\
    &\quad + \sum_{t=0}^{T-1} \Bigg[ C(\mathbf{x}_t, \mathbf{u}_t) -  \sum_{i \in \mathcal{I}_x} \mu_{i}\left( \log(x_{t,i} - \underline{x}_{i}) + \log(\bar{x}_i - x_{t,i}) \right) \nonumber \\
    &\quad\quad -  \sum_{j \in \mathcal{I}_u} \sigma_{j}\left( \log(u_{t,j} - \underline{u}_{j}) + \log(\bar{u}_j - u_{t,j}) \right) + \boldsymbol{\lambda}_{t+1}^{\top} (\mathbf{F}(\mathbf{x}_t, \mathbf{u}_t) - \mathbf{x}_{t+1}) \Bigg].
\end{align}
In the expression above, the variables are defined as follows:
$\boldsymbol{\mu} = [\mu_{1},\mu_{2},\cdots\mu_{n_1}]$ $~~ \text{s.t.}~~\mu_{i} > 0~~\forall i\in1,\cdots n_{1}$ and $\boldsymbol{\sigma} = [\sigma_{1},\sigma_{2},\cdots\sigma_{m_1}]~~ \text{s.t.}~~\sigma_{j} > 0~~\forall j\in 1,\cdots m_{1}$ are the barrier parameters corresponding to the state and control inequality constraints, respectively. 
    $\boldsymbol{\lambda}_t \in \mathbb{R}^n$ for $t=1, \dots, T$ are the Lagrange multipliers, also known as the co-states, associated with the discrete-time dynamic equality constraints at each time step.

The optimal solution is found by solving for the stationary point of $\mathcal{J}_{\boldsymbol{\mu},\boldsymbol{\sigma}}$ with respect to all variables $\{\mathbf{x}_t\}$, $\{\mathbf{u}_t\}$, and $\{\boldsymbol{\lambda}_t\}$ for a given sequence of $\boldsymbol{\mu} \to \boldsymbol{0}$ and $\boldsymbol{\sigma} \to \boldsymbol{0}$.

\section{The Box-iLQR Framework}\label{sec:method}
Having reformulated the optimal control problem as a sequence of unconstrained barrier subproblems in \eqref{eq:augmented_cost}, this section details a solution methodology based on the iterative Linear Quadratic Regulator (iLQR) algorithm. The core idea is to leverage the computational efficiency of iLQR---a second-order method for unconstrained trajectory optimization---by applying it directly to the unconstrained problem structure created by the logarithmic barrier functions.

The Iterative Linear Quadratic Regulator (iLQR) is a prominent trajectory optimization algorithm that solves optimal control problems by iteratively applying a backward-forward pass structure~\cite{ILQR_todorov}. In each iteration, it linearizes the dynamics and quadratizes the cost function around the current nominal trajectory to form a local Linear Quadratic Regulator (LQR) subproblem.

We adapt the iLQR algorithm to solve the barrier-augmented optimization problem defined in~\eqref{eq:augmented_cost}. For the algorithm to be well-posed and converge to a local minimum, we make the following standard assumptions.

\begin{assumption}[Smoothness]
\label{ass:smoothness}
The discrete-time dynamics function $\mathbf{F}(\mathbf{x}_t, \mathbf{u}_t)$ is $\mathcal{C}^1$ (continuously differentiable), and the original cost functions $C(\mathbf{x}, \mathbf{u})$ and $\Phi(\mathbf{x})$ are $\mathcal{C}^2$ (twice continuously differentiable) on their respective domains.
\end{assumption}

\begin{assumption}[Well-Posed Subproblem]
\label{ass:subproblem_minimizer_exists}
For any choice of strictly positive barrier parameters, $\boldsymbol{\mu} > \mathbf{0}$ and $\boldsymbol{\sigma} > \mathbf{0}$, the unconstrained barrier subproblem defined by~\eqref{eq:augmented_cost} is well-posed and admits at least one finite optimal solution.
\end{assumption}

\begin{assumption}[Feasible Initial Trajectory]
\label{ass:feasibility}
A feasible initial trajectory,\\ $(\bar{\mathbf{x}}_0, \dots, \bar{\mathbf{x}}_N; \bar{\mathbf{u}}_0, \dots, \bar{\mathbf{u}}_{N-1})$, is provided as input. This trajectory satisfies all dynamics and inequality constraints.
\end{assumption}
The algorithm is initialized with a feasible trajectory per Assumption~\ref{ass:feasibility}. We note that finding such a trajectory can be a challenging problem in itself, but methods for doing so are outside the scope of this work. The core of the algorithm consists of the backward and forward passes, which we detail below.

\subsubsection{Barrier Cost Terms and Their Derivatives}
Our cost function is augmented with logarithmic barrier terms to enforce inequality constraints. The running cost includes a barrier term $\boldsymbol{\omega}(\mathbf{x}_t, \mathbf{u}_t)$, and the terminal cost includes a barrier term $\boldsymbol{\Omega}(\mathbf{x}_T)$.

The \textbf{running barrier term} is defined as:
\begin{align}
\label{eq:running_barrier}
\boldsymbol{\omega}(\mathbf{x}_t, \mathbf{u}_t) = &- \sum_{i \in \mathcal{I}_x} \mu_{i}\left( \log(x_{t,i} - \underline{x}_{i}) + \log(\overline{x}_i - x_{t,i}) \right) \nonumber\\
&- \sum_{j \in \mathcal{I}_u} \sigma_{j}\left( \log(u_{t,j} - \underline{u}_{j}) + \log(\overline{u}_j - u_{t,j}) \right)
\end{align}
Its first-order derivative (gradient) will have some zero and some non-zero elements. The second-order derivatives (Hessians) is a diagonal matrix, with some zero and some non-zero elements. The non-zero elements of gradient and Hessian are given by:
\begin{align}
\label{eq:log_derivatives}
    &(\boldsymbol{\omega}_{\mathbf{x}})_i = -\mu_i \left( \frac{1}{x_{t,i} - \underline{x}_i} - \frac{1}{\overline{x}_i - x_{t,i}} \right),
     (\boldsymbol{\omega}_{\mathbf{u}})_j = -\sigma_j \left( \frac{1}{u_{t,j} - \underline{u}_j} - \frac{1}{\overline{u}_j - u_{t,j}} \right) \nonumber \\
    &(\boldsymbol{\omega}_{\mathbf{xx}})_{ii} = \mu_i \left( \frac{1}{(x_{t,i} - \underline{x}_i)^2} + \frac{1}{(\overline{x}_i - x_{t,i})^2} \right),\nonumber \\
     &(\boldsymbol{\omega}_{\mathbf{uu}})_{jj} = \sigma_j \left( \frac{1}{(u_{t,j} - \underline{u}_j)^2} + \frac{1}{(\overline{u}_j - u_{t,j})^2} \right),
\end{align}
where $i \in \mathcal{I}_{x}$ and $j \in \mathcal{I}_{u}$. The cross-term Hessian $\boldsymbol{\omega}_{\mathbf{ux}}$ is zero.\\ The \textbf{terminal barrier term} is defined as:\\ $\boldsymbol{\Omega}(\mathbf{x}_T) = - \sum_{i \in \mathcal{I}_x} \mu_{i}\left( \log(x_{T,i} - \underline{x}_{i}) + \log(\overline{x}_i - x_{T,i}) \right)$.
The non-zero elements of its derivatives with respect to the terminal state $\mathbf{x}_T$ are:\\ $ (\boldsymbol{\Omega}_{\mathbf{x}})_i = -\mu_i \left( \frac{1}{x_{T,i} - \underline{x}_i} - \frac{1}{\overline{x}_i - x_{T,i}} \right)$ and $(\boldsymbol{\Omega}_{\mathbf{xx}})_{ii} = \mu_i \left( \frac{1}{(x_{T,i} - \underline{x}_i)^2} + \frac{1}{(\overline{x}_i - x_{T,i})^2} \right)$.

\subsubsection{Backward Pass}
The backward pass computes the local optimal control policy by propagating the coefficients of the updated costate $\boldsymbol{\lambda}_{t}$ = $\mathbf{S}_t\delta \mathbf{x}_{t} + \mathbf{v}_{t}$. It computes the factors, $\mathbf{v}_t$ and $\mathbf{S}_t$, backward in time from $t=T$ to $t=0$.

\noindent\textbf{1. Initialization ($t=T$):}
The value function derivatives are initialized at the terminal time using the derivatives of the total terminal cost, evaluated at the nominal state $\bar{\mathbf{x}}_T$: $ \mathbf{v}_{T} = \Phi_{\mathbf{x}} + \boldsymbol{\Omega}_{\mathbf{x}} \quad \text{and} \quad \mathbf{S}_{T} = \Phi_{\mathbf{xx}} + \boldsymbol{\Omega}_{\mathbf{xx}}$.

\noindent\textbf{2. Backward Recursion ($t=T-1, \dots, 0$):}
For each time step, we compute the derivatives of the action-value function $Q_t$ for deviations around $(\bar{\mathbf{x}}_t, \bar{\mathbf{u}}_t)$:
\begin{align}
    Q_{\mathbf{x}} &= (C_{\mathbf{x}} + \boldsymbol{\omega}_{\mathbf{x}}) + \mathbf{F}_{\mathbf{x}}^{\top}\mathbf{v}_{t+1} \\
    Q_{\mathbf{u}} &= (C_{\mathbf{u}} + \boldsymbol{\omega}_{\mathbf{u}}) + \mathbf{F}_{\mathbf{u}}^{\top}\mathbf{v}_{t+1} \\
    Q_{\mathbf{xx}} &= (C_{\mathbf{xx}} + \boldsymbol{\omega}_{\mathbf{xx}}) + \mathbf{F}_{\mathbf{x}}^{\top}\mathbf{S}_{t+1}\mathbf{F}_{\mathbf{x}} \label{eq:Q_xx}\\
    Q_{\mathbf{uu}} &= (C_{\mathbf{uu}} + \boldsymbol{\omega}_{\mathbf{uu}}) + \mathbf{F}_{\mathbf{u}}^{\top}\mathbf{S}_{t+1}\mathbf{F}_{\mathbf{u}}
    \label{eq:Q_uu}\\
    Q_{\mathbf{ux}} &= C_{\mathbf{ux}} + \mathbf{F}_{\mathbf{u}}^{\top}\mathbf{S}_{t+1}\mathbf{F}_{\mathbf{x}}
\end{align}
All derivatives are evaluated at the nominal trajectory, $(\mathbf{\bar{x}}_{t},\mathbf{\bar{u}}_t)$. The feed-forward term $\mathbf{k}_t$ and feedback gain matrix $\mathbf{K}_t$ are then computed. A regularization term $\zeta \mathbf{I}$ is added to $Q_{\mathbf{uu}}$ to ensure it is positive-definite.
\begin{align}
    \mathbf{k}_{t} &= -(Q_{\mathbf{uu}} + \zeta\mathbf{I})^{-1}Q_{\mathbf{u}} \label{eq:k_update}\\
    \mathbf{K}_{t} &= -(Q_{\mathbf{uu}} + \zeta\mathbf{I})^{-1}Q_{\mathbf{ux}} \label{eq:K_update}
\end{align}
Finally, the values of $\mathbf{S}_t$ and $\mathbf{v}_t$ are updated:
\begin{align}
    \mathbf{v}_{t} &= Q_{\mathbf{x}} + \mathbf{K}_{t}^{\top}Q_{\mathbf{uu}}\mathbf{k}_{t} + \mathbf{K}_{t}^{\top}Q_{\mathbf{u}} + Q_{\mathbf{ux}}^{\top}\mathbf{k}_{t} \label{eq:v_update} \\
    \mathbf{S}_{t} &= Q_{\mathbf{xx}} + \mathbf{K}_{t}^{\top}Q_{\mathbf{uu}}\mathbf{K}_{t} + \mathbf{K}_{t}^{\top}Q_{\mathbf{ux}} + Q_{\mathbf{ux}}^{\top}\mathbf{K}_{t} \label{eq:S_update}
\end{align}

\subsubsection{Forward Pass}
In the forward pass, a new trajectory is generated by applying the computed control policy. A line search is performed with a parameter $\alpha \in (0, 1]$ to ensure improvement in the cost function and maintain feasibility. Starting from $\mathbf{x}_0^{\text{new}} = \bar{\mathbf{x}}_0$:
\begin{align}
    \mathbf{u}_{t}^{\text{new}} &= \bar{\mathbf{u}}_{t} + \alpha \mathbf{k}_{t} + \mathbf{K}_{t}(\mathbf{x}_{t}^{\text{new}} - \bar{\mathbf{x}}_{t}) \quad & t \in [0, T-1] \label{eq:u_update}\\
    \mathbf{x}_{t+1}^{\text{new}} &= \mathbf{F}(\mathbf{x}_{t}^{\text{new}}, \mathbf{u}_{t}^{\text{new}}) \quad & t \in [0, T-1]
\end{align}
The new trajectory $(\mathbf{x}^{\text{new}}, \mathbf{u}^{\text{new}})$ becomes the nominal trajectory for the next iteration. This forward-backward procedure is repeated until convergence. The backtracking line search on $\alpha$ ensures that the augmented cost function $\mathcal{J}_{\boldsymbol{\mu},\boldsymbol{\sigma}}$ decreases at each iteration. A new trajectory is accepted if it provides a sufficient decrease, typically measured by comparing the actual reduction in cost to the reduction predicted by the local quadratic model. This ensures the algorithm makes progress towards a local minimum. A common acceptance criterion is when the ratio of actual to expected reduction is positive:
\begin{equation}
\label{eq:ilqr_reduction_ratio}
    \frac{\mathcal{J}(\text{new}) - \mathcal{J}(\text{old})}{\Delta J_{\text{expected}}} < c_1
\end{equation}
where $\mathcal{J}(\text{old})$ and $\mathcal{J}(\text{new})$ are the costs of the old and new trajectories, $c_1 \in (0, 1)$ is a small constant (e.g., $10^{-4}$), and $\Delta J_{\text{expected}} < 0$ is the cost change predicted by the quadratic approximation.
\begin{equation}\label{eq:expected_Cost_reduction}
    \Delta J_{\text{expected}} = -\Big(\alpha - \frac{\alpha^2}{2}\Big)\sum_{i=0}^{T-1}Q_{\mathbf{u}}^{\top}Q_{\mathbf{uu}}^{-1}Q_{\mathbf{u}}.
\end{equation}
One should note here that we need $Q_{\mathbf{uu}}$ to be positive definite for $\Delta J_{\text{expected}} < 0$, i.e, the expected cost to decrease. 

\subsubsection{Barrier Relaxation}
The Box-iLQR method presented in the previous section solves the optimization problem for a \textit{fixed} set of barrier parameters $\boldsymbol{\mu}$ and $\boldsymbol{\sigma}$. To recover the solution to the original constrained problem, these parameters must be driven to zero. This is managed by an outer loop that iteratively reduces the barrier parameters and uses the iLQR solver to find the solution for each new subproblem. This sequence of solutions ideally follows the \textit{central path} towards the constrained optimum. The complete procedure, which we refer to as the Box-iLQR algorithm, is detailed in Algorithm~\ref{alg:barrier_update}. The process begins with conservatively large initial barrier parameters, $\boldsymbol{\mu}_0$ and $\boldsymbol{\sigma}_0$, and a feasible initial trajectory. The algorithm then enters an inner loop where it first solves the iLQR subproblem for the current parameters, using the solution from the previous iteration as a warm start to accelerate convergence. 

A crucial aspect of this process is the parameter update rule. After a successful inner-loop solve, each component of $\boldsymbol{\mu}$ and $\boldsymbol{\sigma}$ is decreased by a multiplicative factor. A naive or overly aggressive reduction can push the warm-start trajectory into an infeasible region where the logarithmic barriers are undefined, causing the inner-loop solver to fail. To handle this, our update schedule is adaptive. If the iLQR solver fails to converge, it indicates that the barrier parameters were reduced too severely. In response, the algorithm reverts the parameter update and applies a more conservative reduction factor for the corresponding constraint in the next attempt. This process continues until the barrier parameters are below a specified tolerance, ensuring their contribution to the total cost is negligible and the solution closely approximates the true constrained optimum.

\begin{remark}[Model-Based Formulation]
The presented algorithm is model-based, as it requires access to the gradients of the dynamics function, specifically the Jacobians $\mathbf{F}_{\mathbf{x}} = \frac{\partial \mathbf{F}}{\partial \mathbf{x}}$ and $\mathbf{F}_{\mathbf{u}} = \frac{\partial \mathbf{F}}{\partial \mathbf{u}}$. When an analytical expression for the dynamics $\mathbf{F}$ is available, these Jacobians can be derived directly or computed efficiently using automatic differentiation tools.
Conversely, in a model-free setting where $\mathbf{F}$ is treated as a black box, one must resort to alternative methods. For instance, model-free variants of iLQR learn local linear models of the dynamics at each iteration from data gathered by interacting with the system~\cite{parunandi2020d2c,rand2c}.
\end{remark}

\section{Algorithmic Properties and Advantages}\label{sec:alg_properties}
We move on to demonstrate that the algorithm described in the previous section is not merely a convenience, but offers significant advantages. We show that the Hessian of the logarithmic barrier term serves as a natural regularizer enhancing its numerical stability compared to the vanilla formulation. Furthermore, we also show that the feedback gain matrices obtained close to the constraint boundaries have some significant qualities, which we call ``constraint-aware feedback''. The log-barrier formulation naturally encodes constraint information into the feedback gains, causing them to diminish as the state or control approaches a constraint boundary. This effectively yields a constraint-respecting feedback law, which is also suitable for deployment in the presence of uncertainty. To substantiate these claims, the following subsections provide the advantages, and accompanying mathematical proofs, and analysis.

\begin{algorithm}
\caption{Box-iLQR}
\label{alg:barrier_update}
\begin{algorithmic}[1]
\Require Initial feasible trajectory $(\bar{\mathbf{x}}, \bar{\mathbf{u}})$
\Require Initial barrier parameters $\boldsymbol{\mu}_0, \boldsymbol{\sigma}_0$
\Require Initial reduction factors $\mathbf{r}_{\mu} \in (0,1)^{n_1}, \mathbf{r}_{\sigma} \in (0,1)^{m_1}$
\Require Reduction factor update rate $\beta_r > 1$ (e.g., 1.5)
\Require Barrier termination tolerance $\epsilon_{\text{barrier}}$

\State Initialize $\boldsymbol{\mu} \leftarrow \boldsymbol{\mu}_0$, $\boldsymbol{\sigma} \leftarrow \boldsymbol{\sigma}_0$
\State Initialize $\boldsymbol{\mu}_{\text{prev}} \leftarrow \boldsymbol{\mu}$, $\boldsymbol{\sigma}_{\text{prev}} \leftarrow \boldsymbol{\sigma}$

\While{$\|\boldsymbol{\mu}\| > \epsilon_{\text{barrier}}$ or $\|\boldsymbol{\sigma}\| > \epsilon_{\text{barrier}}$}
    \State \Comment{Solve the inner-loop problem with a warm start.}
    \State $(\bar{\mathbf{x}}, \bar{\mathbf{u}}), \text{success}, \text{failed\_idx} \leftarrow \Call{iLQR\_Solve}{\bar{\mathbf{x}}, \bar{\mathbf{u}}, \boldsymbol{\mu}, \boldsymbol{\sigma}}$
    
    \If{\textbf{not} success}
        \State \Comment{Reduction was too aggressive; revert and adapt.}
        \State $\boldsymbol{\mu} \leftarrow \boldsymbol{\mu}_{\text{prev}}$, $\boldsymbol{\sigma} \leftarrow \boldsymbol{\sigma}_{\text{prev}}$
        \If{constraint $i$ in $\boldsymbol{\mu}$ failed (from $\text{failed\_idx}$)}
            \State $(\mathbf{r}_{\mu})_i \leftarrow \min(1, (\mathbf{r}_{\mu})_i \cdot \beta_r)$ \Comment{Make reduction less aggressive.}
        \ElsIf{constraint $j$ in $\boldsymbol{\sigma}$ failed}
            \State $(\mathbf{r}_{\sigma})_j \leftarrow \min(1, (\mathbf{r}_{\sigma})_j \cdot \beta_r)$
        \EndIf
        \State \textbf{continue} \Comment{Retry solving with old parameters and new rates.}
    \Else
        \State \Comment{Successful solve; prepare for next reduction.}
        \State $\boldsymbol{\mu}_{\text{prev}} \leftarrow \boldsymbol{\mu}$, $\boldsymbol{\sigma}_{\text{prev}} \leftarrow \boldsymbol{\sigma}$
        \State $\boldsymbol{\mu} \leftarrow \mathbf{r}_{\mu} \odot \boldsymbol{\mu}$ \Comment{Element-wise multiplication.}
        \State $\boldsymbol{\sigma} \leftarrow \mathbf{r}_{\sigma} \odot \boldsymbol{\sigma}$
    \EndIf
\EndWhile

\State \Return Optimal trajectory $(\bar{\mathbf{x}}, \bar{\mathbf{u}})$
\end{algorithmic}
\end{algorithm}

\subsection{Inherent Regularization and Stability}
The stability of the iLQR backward pass hinges on the positive definiteness of the matrix $Q_{\mathbf{uu}}$, as this ensures a guaranteed cost reduction per eq.~\eqref{eq:expected_Cost_reduction}. Standard iLQR enforces this condition by adding an explicit regularization term, such as $\zeta \mathbf{I}$ (cf. eqs.~\eqref{eq:k_update} and \eqref{eq:K_update}). Our log-barrier formulation introduces an inherent regularization that enhances numerical stability through two mechanisms.

First, the term $\boldsymbol{\omega}_{\mathbf{uu}}$ from eq.~\eqref{eq:log_derivatives}, representing the second derivative of the control barrier, is added directly to the $Q_{\mathbf{uu}}$ expression. As the Hessian of a convex function, $\boldsymbol{\omega}_{\mathbf{uu}}$ is positive definite and acts as an adaptive regularization term analogous to the manually tuned $\zeta \mathbf{I}$.

Second, the state log barriers provide an indirect regularization effect. The term $\boldsymbol{\omega}_{\mathbf{xx}}$ makes the value function Hessian from the next time step, $\mathbf{S}_{t+1}$, better conditioned through direct addition in the $Q_{\mathbf{xx}}$ term (eq.~\eqref{eq:Q_xx}). This improved conditioning of $\mathbf{S}_{t+1}$ propagates through the $Q_{\mathbf{uu}}$ update in eq.~\eqref{eq:Q_uu}, further safeguarding its positive definiteness. Collectively, this dual-regularization effect stabilizes the backward pass, particularly in regions where the unregularized $Q_{\mathbf{uu}}$ would otherwise be ill-conditioned or indefinite.

\subsection{Convergence Guarantee}
\begin{assumption}[Positive-definiteness of the control Hessian]\label{ass:pd_Quu}
    The control Hessian term, $Q_{\mathbf{uu}}$ in eq.~\eqref{eq:Q_uu}, is positive definite at all time steps.
\end{assumption}

\begin{theorem}[Convergence to the optima]\label{thm:convergence}
    Under assumptions \ref{ass:smoothness}-\ref{ass:pd_Quu}, the Box-iLQR Algorithm~\ref{alg:barrier_update} will always converge to an optimal solution for each $(\bmu,\bsigma)$ pair and reach an optimal solution of the optimal control problem in eq.~\eqref{eq:ocp_dis} as $ (\bmu,\bsigma) \to \boldsymbol{0}$.
\end{theorem}

\begin{proof}
Let us analyze the change in the total cost, $\mathcal{J}$, resulting from the control update. We begin by considering the expected cost change as predicted by the quadratic model from eq.~\eqref{eq:expected_Cost_reduction}:
\begin{equation*}
    \Delta J_{\text{expected}} = -\Big(\alpha - \frac{\alpha^2}{2}\Big)\sum_{i=0}^{T-1}Q_{\mathbf{u}}^{\top}Q_{\mathbf{uu}}^{-1}Q_{\mathbf{u}}.
\end{equation*}
Under Assumption~\ref{ass:pd_Quu} ($Q_{\mathbf{uu}}$ is positive definite) and for a line search parameter $\alpha \in (0,1)$, the term $(\alpha - \alpha^2/2)$ is positive. Thus, $\Delta J_{\text{expected}} \leq 0$. The iLQR backward pass therefore always predicts a cost reduction, unless the solution has converged ($Q_{\mathbf{u}} = \mathbf{0}$). For the true nonlinear cost, $\mathcal{J}$, to decrease, the quadratic approximation must be sufficiently valid. This can be ensured by a suitable choice of $\alpha$.

To demonstrate this, we first analyze the order of magnitude of the state and control perturbations. The control update rule from eq.~\eqref{eq:u_update} is given by $\delta \mathbf{u}_{t} = \alpha \mathbf{k}_{t} + \mathbf{K}_{t}\delta \mathbf{x}_{t}$. We analyze the trajectory rollout starting from the initial state $\mathbf{x}_0$, which is fixed, implying $\delta \mathbf{x}_{0} = \mathbf{0}$.

At time step $t=0$:
\begin{align*}
    \delta \mathbf{u}_0 &= \alpha \mathbf{k}_0 + \mathbf{K}_0 \delta \mathbf{x}_0 = \alpha \mathbf{k}_0 = \mathcal{O}(\alpha).
\end{align*}
The state perturbation at the next step, $\delta \mathbf{x}_1$, is found by expanding the dynamics:
\begin{align*}
    \delta \mathbf{x}_{1} =& \mathbf{F}_{\mathbf{x}}^{0} \delta \mathbf{x}_{0} + \mathbf{F}_{\mathbf{u}}^{0} \delta \mathbf{u}_{0} \\
    &+ \frac{1}{2}  \delta \mathbf{x}_{0}^{\top}\mathbf{F}_{\mathbf{xx}}^{0}\delta \mathbf{x}_{0} + \delta \mathbf{u}_{0}\mathbf{F}_{\mathbf{ux}}^{0}\delta \mathbf{x}_{0} + \frac{1}{2}\delta \mathbf{u}_{0}^{\top}\mathbf{F}_{\mathbf{uu}}^{0}\delta \mathbf{u}_{0} + \mathcal{O}(\| \delta \mathbf{x}_{0}\|^{3},\| \delta \mathbf{u}_{0}\|^{3}).
\end{align*}
Using $\delta \mathbf{x}_{0} = \boldsymbol{0}$ and $\delta \mathbf{u}_{0} = \mathcal{O}(\alpha)$, we get
\begin{align*}
    \delta \mathbf{x}_{1} = \mathbf{F}_{\mathbf{u}}^{0}\delta \mathbf{u}_{0} + \frac{1}{2} \delta \mathbf{u}_{0}^{\top} \mathbf{F}_{\mathbf{uu}}^{0}\delta \mathbf{u}_{0} + \mathcal{O}(\alpha^{3}),
\end{align*}
which gives
\begin{align*}
    \delta \mathbf{x}_{1} = \alpha \mathbf{F}_{\mathbf{u}}^{0}\mathbf{k}_{0} + \frac{\alpha^2}{2}\mathbf{k}_{0}^{\top}\mathbf{F}_{\mathbf{uu}}\mathbf{k}_{0} + \mathcal{O}(\alpha^{3}).
\end{align*}
Similarly at the next time-step, we have
\begin{align*}
    \delta \mathbf{x}_{2} =& \mathbf{F}_{\mathbf{x}}^{1} \delta \mathbf{x}_{1} + \mathbf{F}_{\mathbf{u}}^{1} \delta \mathbf{u}_{1} \\
    &+ \frac{1}{2}  \delta \mathbf{x}_{1}^{\top}\mathbf{F}_{\mathbf{xx}}^{1}\delta \mathbf{x}_{1} + \delta \mathbf{u}_{1}\mathbf{F}_{\mathbf{ux}}^{1}\delta \mathbf{x}_{1} + \frac{1}{2}\delta \mathbf{u}_{1}^{\top}\mathbf{F}_{\mathbf{uu}}^{1}\delta \mathbf{u}_{1} + \mathcal{O}(\| \delta \mathbf{x}_{1}\|^{3},\| \delta \mathbf{u}_{1}\|^{3}).
\end{align*}
Substituting for $\delta \mathbf{x}_{1}$, we get
\begin{align*}
    \delta \mathbf{x}_{2} =& \mathbf{F}_{\mathbf{x}}^{1}\big( \alpha \mathbf{F}_{\mathbf{u}}^{0}\mathbf{k}_{0} + \frac{\alpha^2}{2}\mathbf{k}_{0}^{\top}\mathbf{F}_{\mathbf{uu}}\mathbf{k}_{0} + \mathcal{O}(\alpha^{3}) \big) \\
    &+ \alpha \mathbf{F}_{\mathbf{u}}^{1}\mathbf{k}_{1} + \mathbf{F}_{\mathbf{u}}^{1}\mathbf{K}_{1}\big( \alpha \mathbf{F}_{\mathbf{u}}^{0}\mathbf{k}_{0} + \frac{\alpha^2}{2}\mathbf{k}_{0}^{\top}\mathbf{F}_{\mathbf{uu}}\mathbf{k}_{0} + \mathcal{O}(\alpha^{3}) \big )  + \mathcal{O}(\alpha^2),
\end{align*}
which can be simplified to
\begin{align*}
    \delta \mathbf{x}_{2} = \alpha \mathbf{F}_{\mathbf{u}}^{1}\mathbf{k}_{1} + \alpha \big(  \mathbf{F}_{\mathbf{x}}^{1} + \mathbf{F}_{\mathbf{u}}^{1}\mathbf{K}_{1}\big) \mathbf{F}_{\mathbf{u}}^{0}\mathbf{k}_{0} + \mathcal{O}(\alpha^2).
\end{align*}
So, for a given time step $t$, $\delta \mathbf{x}_{t}$ is:
\begin{align*}
    \delta \mathbf{x}_t =& \alpha \mathbf{F}_{\mathbf{u}}^{t-1}\mathbf{k}_{t-1} + \alpha \big( \mathbf{F}_{\mathbf{x}}^{t-1} + \mathbf{F}_{\mathbf{u}}^{t-1} \mathbf{K}_{t-1}\big) \mathbf{F}_{\mathbf{u}}^{t-2}\mathbf{k}_{t-2} \\
    &+ \alpha \big( \mathbf{F}_{\mathbf{x}}^{t-1} + \mathbf{F}_{\mathbf{u}}^{t-1} \mathbf{K}_{t-1}\big)\big( \mathbf{F}_{\mathbf{x}}^{t-2} + \mathbf{F}_{\mathbf{u}}^{t-2} \mathbf{K}_{t-2}\big) \mathbf{F}_{\mathbf{u}}^{t-3}\mathbf{k}_{t-3} + \cdots\\
    &+\alpha \big( \mathbf{F}_{\mathbf{x}}^{t-1} + \mathbf{F}_{\mathbf{u}}^{t-1} \mathbf{K}_{t-1}\big) \cdots \cdots \mathbf{F}_{\mathbf{u}}^{0}\mathbf{k}_{0} + \alpha^{2}(\cdots) + \alpha^{3} (\cdots) + \cdots.
\end{align*}
From the above equation, it can be stated that the first-order terms in the dynamics have a coefficient of $\alpha$, the second-order terms have $\alpha^{2}$, and so on. This demonstrates that choosing a small $\alpha$ confines the perturbations to be small along the entire trajectory, which is the key to ensuring the validity of local linear approximations. The higher-order terms in dynamics, such as the second order terms, are of order $\mathcal{O}(\alpha^2)$ and diminish faster.

Now, we consider the actual change in cost, $\Delta \mathcal{J} = \mathcal{J}_{\text{new}} - \mathcal{J}_{\text{old}}$. Taylor series expansion implies that the actual cost change can be related to the quadratically predicted change, $\Delta J_{\text{expected}}$. The discrepancy arises from two sources: (i) higher-order terms in the cost function (third-order and above), and (ii) the error between the nonlinear dynamics and the linear model used in the backward pass.
The higher-order cost terms are of the form $\mathcal{O}(\|\delta\mathbf{x}\|^3, \|\delta\mathbf{u}\|^3)$, which, based on our analysis, are $\mathcal{O}(\alpha^3)$. The effect of the dynamics' nonlinearities, which are of order $\mathcal{O}(\alpha^2)$, are integrated over the trajectory and multiplied by the co-state $\boldsymbol{\lambda}$. All other dynamics' nonlinearities lead to an error of order $\mathcal{O}(\alpha^3)$. The resulting error is of order $\mathcal{O}(\alpha^2)$.
Therefore, the actual cost change can be expressed as:
\begin{equation*}
    \Delta \mathcal{J} = \Delta J_{\text{expected}} + \mathcal{O}(\alpha^2).
\end{equation*}
The line search acceptance criterion from eq.~\eqref{eq:ilqr_reduction_ratio} requires the ratio $\rho =$\\ $\Delta \mathcal{J} / \Delta J_{\text{expected}}$ to be greater than some parameter $c_1 \in (0,1)$. Let us examine this ratio in the limit as $\alpha \to 0$:
\begin{align*}
    \lim_{\alpha \to 0} \frac{\Delta \mathcal{J}}{\Delta J_{\text{expected}}} &= \lim_{\alpha \to 0} \frac{-\big(\alpha - \frac{\alpha^2}{2}\big)\sum Q_{\mathbf{u}}^{\top}Q_{\mathbf{uu}}^{-1}Q_{\mathbf{u}} + \mathcal{O}(\alpha^2)}{-\big(\alpha - \frac{\alpha^2}{2}\big)\sum Q_{\mathbf{u}}^{\top}Q_{\mathbf{uu}}^{-1}Q_{\mathbf{u}}} \\
    &= \lim_{\alpha \to 0} \frac{-\alpha \sum Q_{\mathbf{u}}^{\top}Q_{\mathbf{uu}}^{-1}Q_{\mathbf{u}} + \mathcal{O}(\alpha^2)}{-\alpha \sum Q_{\mathbf{u}}^{\top}Q_{\mathbf{uu}}^{-1}Q_{\mathbf{u}} + \mathcal{O}(\alpha^2)} = 1.
\end{align*}
Since the ratio of the actual to the expected cost reduction approaches 1 as $\alpha \to 0$, for any $c_1 \in (0,1)$, there must exist a sufficiently small $\alpha > 0$ such that $\Delta \mathcal{J} / \Delta J_{\text{expected}} > c_1$. This guaranties that the line search will eventually find an acceptable step that reduces the true cost $\mathcal{J}$, provided that the algorithm has not converged ($Q_{\mathbf{u}} \neq \mathbf{0}$). The cost reduction is thus guaranteed with $\mathcal{J}_{\text{new}} < \mathcal{J}_{\text{old}}$. The equality $\mathcal{J}_{\text{new}} = \mathcal{J}_{\text{old}}$ holds only at the convergence where $\Delta J_{\text{expected}} = 0$. Having established the convergence of Box-iLQR to an optimum for a fixed set of parameters $(\boldsymbol{\mu},\boldsymbol{\sigma})$, we now proceed to demonstrate that constraint satisfaction is maintained as these parameters are gradually reduced within the algorithm's outer loop.

Having established the convergence of the inner iLQR loop to a local minimum for a fixed set of parameters $(\bmu, \bsigma)$, we now prove that strict constraint satisfaction is maintained as these parameters are reduced in the algorithm's outer loop. 

We begin by defining the relevant quantities at a given iteration of the algorithm. Let $\mathbf{\Gamma} = [\bmu, \bsigma]^{\top}$ denote the vector of all barrier parameters. At each time step $t$, the concatenated vector of optimal state and control variables is denoted by $\mathbf{z}_{t} = [\mathbf{x}_{t}, \mathbf{u}_{t}]^{\top}$.

The slackness of $\mathbf{z}_{t}$ with respect to its lower and upper bounds, $\underline{\mathbf{z}}_{t}$ and $\bar{\mathbf{z}}_{t}$, is given by the vectors $\underline{\boldsymbol{\varepsilon}}_t$ and $\bar{\boldsymbol{\varepsilon}}_t$, respectively. These quantities are defined by the relations:
\begin{align}
    \mathbf{z}_{t} - \underline{\mathbf{z}}_{t} &= \underline{\boldsymbol{\varepsilon}}_t, \\
    \bar{\mathbf{z}}_{t} - \mathbf{z}_{t} &= \bar{\boldsymbol{\varepsilon}}_t.
\end{align}
For any unconstrained component of $\mathbf{z}_{t}$, the corresponding bound is infinite (i.e., $\underline{z}_{t,i} = -\infty$ or $\bar{z}_{t,i} = \infty$), and the associated slackness is also infinite ($\underline{\varepsilon}_{t,i} = \infty$ or $\bar{\varepsilon}_{t,i} = \infty$).

To ensure that the updated variables remain strictly feasible in the subsequent iteration, the change $\delta \mathbf{z}_{t}$ must satisfy:
\begin{equation}
    \underline{\mathbf{z}}_{t} < \mathbf{z}_{t} + \delta \mathbf{z}_{t} < \bar{\mathbf{z}}_{t}.
\end{equation}
The update $\delta \mathbf{z}_{t}$ is obtained from a quadratic approximation of the augmented cost function $\mathcal{J}$ from \eqref{eq:augmented_cost} following a change in the barrier parameters from $\mathbf{\Gamma}^{\text{iter}}$ to $\mathbf{\Gamma}^{\text{iter+1}}$. Using a line search parameter $\alpha$, this update is approximated by:
\begin{equation*}
    \delta \mathbf{z}_{t} \approx -\alpha \mathcal{J}_{\mathbf{z}_t,\mathbf{z}_t}^{-1}(\Delta \mathcal{J}_{\mathbf{z}_t}),
\end{equation*}
where $\Delta \mathcal{J}_{\mathbf{z}_{t}} = \mathcal{J}_{\mathbf{z}_t}^{\text{iter+1}} - \mathcal{J}_{\mathbf{z}_t}^{\text{iter}}$ is the difference between the cost gradients over successive iterations, and $\mathcal{J}_{\mathbf{z}_t,\mathbf{z}_t}$ is the Hessian at the current iteration and is considered to be invertible.

To maintain feasibility with respect to the lower bound, we require $\mathbf{z}_{t} + \delta \mathbf{z}_{t} > \underline{\mathbf{z}}_t$. Substituting the expressions for $\mathbf{z}_t$ and $\delta \mathbf{z}_t$ yields:
\begin{align*}
    (\underline{\mathbf{z}}_{t} + \underline{\boldsymbol{\varepsilon}}_t) - \alpha \mathcal{J}_{\mathbf{z}_t,\mathbf{z}_t}^{-1}(\Delta \mathcal{J}_{\mathbf{z}_t}) &> \underline{\mathbf{z}}_t \\
    \underline{\boldsymbol{\varepsilon}}_t - \alpha \mathcal{J}_{\mathbf{z}_t,\mathbf{z}_t}^{-1}(\Delta \mathcal{J}_{\mathbf{z}_t}) &> \mathbf{0} \\
    \alpha \mathcal{J}_{\mathbf{z}_t,\mathbf{z}_t}^{-1}(\Delta \mathcal{J}_{\mathbf{z}_t}) &< \underline{\boldsymbol{\varepsilon}}_t.
\end{align*}
This inequality dictates the maximum allowable value for $\alpha$. The constraint is trivially satisfied for unconstrained variables, as their corresponding slackness $\underline{\varepsilon}_{t,s}$ is infinite. Furthermore, for any component $s$ where the term $\big(\mathcal{J}_{\mathbf{z}_t,\mathbf{z}_t}^{-1}(\Delta \mathcal{J}_{\mathbf{z}_t})\big)_s$ is negative, the inequality holds for any positive $\alpha$. 

Therefore, the choice of $\alpha$ is only constrained by the components where $\big(\mathcal{J}_{\mathbf{z}_t,\mathbf{z}_t}^{-1}(\Delta \mathcal{J}_{\mathbf{z}_t})\big)_s > 0$. Let $(\mathcal{J}_{\mathbf{z}_t,\mathbf{z}_t}^{-1}(\Delta \mathcal{J}_{\mathbf{z}_t}))^{+}$ denote the vector containing only these positive elements for the constrained variables, and let $\underline{\boldsymbol{\varepsilon}}_{t}^{+}$ be the vector of the corresponding slack variables. To ensure no constraint is violated, $\alpha$ must satisfy the condition for all such components simultaneously. This leads to the bound:
\begin{equation}
    \alpha < \min_{s} \Bigg[ \frac{\underline{\boldsymbol{\varepsilon}}_{t,s}^{+}}{\big( (\mathcal{J}_{\mathbf{z}_t,\mathbf{z}_t}^{-1}(\Delta \mathcal{J}_{\mathbf{z}_t}))^{+}\big)_{s}}\Bigg],
    \label{eq:alpha_low}
\end{equation}
where $(\cdot)_{s}$ denotes the $s$-th element of the corresponding vector.

A similar derivation for the upper bound feasibility condition, $\mathbf{z}_{t} + \delta \mathbf{z}_{t} < \bar{\mathbf{z}}_{t}$, yields an analogous constraint on $\alpha$:
\begin{equation}
    \alpha < \min_{s}\Bigg[ \frac{\bar{\boldsymbol{\varepsilon}}_{t,s}^{+}}{\big( (\mathcal{J}_{\mathbf{z}_t,\mathbf{z}_t}^{-1}(\Delta \mathcal{J}_{\mathbf{z}_t}))^{+}\big)_{s}} \Bigg].
    \label{eq:alpha_high}
\end{equation}
To ensure that all constraints are satisfied at every time step $t$, and to also satisfy the sufficient decrease condition from \eqref{eq:ilqr_reduction_ratio}, the final line search parameter must be chosen as:
\begin{equation}
    \alpha < \min_{t,s}\Bigg[ \frac{\underline{\boldsymbol{\varepsilon}}_{t,s}^{+}}{\big( (\mathcal{J}_{\mathbf{z}_t,\mathbf{z}_t}^{-1}(\Delta \mathcal{J}_{\mathbf{z}_t}))^{+}\big)_{s}}, ~ \frac{\bar{\boldsymbol{\varepsilon}}_{t,s}^{+}}{\big( (\mathcal{J}_{\mathbf{z}_t,\mathbf{z}_t}^{-1}(\Delta \mathcal{J}_{\mathbf{z}_t}))^{+}\big)_{s}},~\alpha_{\text{LQ}}\Bigg],
    \label{eq:line_search}
\end{equation}
where $\alpha_{\text{LQ}}$ denote the value of the line search parameter such that the inequality in eq.~\eqref{eq:ilqr_reduction_ratio} is satisfied. The change in the cost gradient, $\Delta \mathcal{J}_{\mathbf{z}_t}$, is driven by the change in the barrier parameters. For the $p$-th component, this relationship is given by:
\begin{equation*}
    (\Delta \mathcal{J}_{\mathbf{z}_t})_{s} = \big( \gamma_{s}^{\text{iter}} - \gamma_{s}^{\text{iter+1}} \big)\left[ \frac{1}{\underline{\varepsilon}_{t,s}} - \frac{1}{\bar{\varepsilon}_{t,s}}\right],
\end{equation*}
where $\gamma_{s}^{p}$ is the $s$-th component of $\mathbf{\Gamma}$ at $p$-th iteration. If either $\underline{\varepsilon}_{t,s} <<1$ or $\bar{\varepsilon}_{t,s} <<1$, the variable $z_{t,s}$ is very close to either its lower bound or its upper bound, $\mathcal{J}_{\mathbf{z}_t,\mathbf{z}_t}^{-1}(\Delta \mathcal{J}_{\mathbf{z}_t})$ is dominated by the barrier term and can be approximated as
\begin{align}    (\mathcal{J}_{\mathbf{z}_t,\mathbf{z}_t}^{-1}(\Delta \mathcal{J}_{\mathbf{z}_t}))_{s} \approx \Bigg( \frac{\gamma_{s}^{\text{iter+1}}}{\varepsilon_{t,s}^{2}}\Bigg)^{-1} \Bigg( \frac{\big( \gamma_{s}^{\text{iter}} - \gamma_{s}^{\text{iter+1}} \big)}{\varepsilon_{t,s}} \Bigg) = \frac{\gamma_{s}^{\text{iter}} - \gamma_{s}^{\text{iter+1}}}{\gamma_{s}^{\text{iter+1}}}\varepsilon_{t,s}.
\end{align}
The $\varepsilon_{t,s}$ in the above equation will cancel out with the numerator if we substitute it in eq.~\eqref{eq:line_search} and the corresponding requirement on $\alpha$ becomes
\begin{equation}
    \alpha < \frac{\gamma_{s}^{\text{iter+1}}}{\gamma_{s}^{\text{iter}} - \gamma_{s}^{\text{iter+1}}}.
\end{equation}
Thus, the bound on $\alpha$ never becomes $0$ and remains a positive value. This means that there exist an $\alpha \in (0,1)$, such that box constraints in eq.~\eqref{eq:ocp_x_bounds_dis} and eq.~\eqref{eq:ocp_u_bounds_dis} will not be violated as barrier parameters, $(\bmu,\bsigma)\to \boldsymbol{0}$.

Equation~\eqref{eq:line_search} reveals an inverse relationship between the step size $\alpha$ and the magnitude of the change in the barrier parameters, $(\gamma^{\text{iter}} - \gamma^{\text{iter+1}})$, since this term appears in the denominator via $\Delta \mathcal{J}_{\mathbf{z}_t}$. This signifies that an aggressive reduction in the barrier parameters necessitates a smaller step size $\alpha$ to maintain feasibility, potentially slowing convergence. Conversely, the equation shows that $\alpha$ is directly proportional to the slackness $\varepsilon$, allowing for larger steps when the variables are far from their bounds.

Hence, by choosing a suitable line search parameter, $\alpha$, we can say that Box-iLQR converges to the optimal solution for a given pair of $(\bmu,\bsigma)$ and respects the constraints when the barrier parameters are reduced. So, Box-iLQR converges to an optimal solution as  $(\bmu,\bsigma) \to \boldsymbol{0}$.

\end{proof}

From theorem~\ref{thm:convergence}, it is inferred that Box-iLQR preserves convexity if the problem in eq.~\eqref{eq:ocp_dis} is already convex. 
\begin{remark}
    Under the conditions of convexity of the Hamiltonian ~\cite{abhi_convexity} and subject to Assumptions \ref{ass:smoothness}--\ref{ass:feasibility}, the sequence of solutions generated by the Box-iLQR algorithm converges to the global optimal solution in the limit where the parameters $\boldsymbol{\mu},\boldsymbol{\sigma}\to \boldsymbol{0}$. 
\end{remark}

This regularization and preservation of structure provides a significant advantage over methods~\cite{CL_DDP,lantoine2012hybrid,DDP_IPM,howell2019altro}, as they might need ad hoc regularization and also converge to a suboptimal solution.
\subsection{Constraint-Aware Feedback Control}
\label{subsec:constraint-aware}
The iLQR algorithm updates the control using a feedforward term ($\mathbf{k}_t$) and a feedback gain matrix ($\mathbf{K}_t$). This section focuses on the feedback term, which serves two crucial functions. First, during optimization, it ensures iterative updates remain feasible, preventing constraint violations. Second, at convergence, the optimal gains $\mathbf{K}_t^*$ form a local, Linear Time-Varying (LTV) feedback controller for tracking the nominal trajectory amidst disturbances~\cite{rand2c,naveed_feedback}. The feedback law is given by:
\begin{equation} 
\label{eq:feedback_law}
    \mathbf{u}_{t} = \mathbf{u}_t^* + \mathbf{K}_t^*(\mathbf{x}_t - \mathbf{x}_t^*)
\end{equation}
where a corrective action is applied based on the deviation from the optimal state $\mathbf{x}_t^*$.

We will prove that the feedback controller derived from our Box-iLQR is inherently ``constraint-aware'' in two ways:
    1. When a control input $u_{t,j}$ is at its boundary (eq.~\eqref{eq:ocp_u_bounds_dis}), the corresponding row of the feedback matrix, $(\mathbf{K}_t)_j$, becomes a zero vector. This prevents the feedback action from commanding an infeasible change. and
    2. The barrier terms induce steep curvatures in the local cost-to-go ($Q_{\mathbf{uu}}, Q_{\mathbf{ux}}$), resulting in feedback gains that aggressively counteract any state deviation towards an infeasible region (eq.~\eqref{eq:ocp_x_bounds_dis}).

We now proceed to formalize and prove these claims.

\begin{theorem}[Feedback Gain Structure at Control Boundaries]
\label{thm:gain_structure_control}
The feedback controller in eq.~\eqref{eq:feedback_law} inherently respects the control saturation limits from eq.~\eqref{eq:control_box_constraints_idx}. Specifically, if an optimal control input $u_{t,j}$ is saturated at its boundary, then the corresponding row of the feedback matrix $\mathbf{K}_t$ is the zero vector.
\end{theorem}
\begin{proof}
    To begin with, let us revisit the feedback equation:
    \begin{equation}
    \label{eq:optimal_feedback}
        \mathbf{K}_t = -(Q_{\mathbf{uu}}+\zeta \mathbf{I})^{-1}Q_{\mathbf{ux}},
    \end{equation}
    and say
    \begin{equation}
    \label{eq:Quu_th1}
        \mathbf{M} = Q_{\mathbf{uu}} +\zeta \mathbf{I} = C_{\mathbf{uu}} + \boldsymbol{\omega}_{\mathbf{uu}} + \mathbf{F}_{\mathbf{u}}^{\top}\mathbf{S}_{t+1}\mathbf{F}_{\mathbf{u}} +\zeta \mathbf{I}, 
    \end{equation}
    where $(\boldsymbol{\omega}_{\mathbf{uu}})_{jj}$ is given by eq. \eqref{eq:log_derivatives}. It is easy to see that whenever a control nears saturation, the corresponding diagonal term of $(\boldsymbol{\omega}_{\mathbf{uu}})$ tends to $\infty$. Let $\mathcal{I}_{u,c} \subseteq \mathcal{I}_{u}$ be the set of indices corresponding to the saturated control input and the cardinality of $|\mathcal{I}_{u,c}| = m_{1,c}$, where $1\leq m_{1,c}\leq m_{1}$. Now, we can say that $(\mathbf{M})_{kk} \to \infty \quad \forall ~k\in\mathcal{I}_{u,c}$, i.e, the $k$th diagonal elements of $\mathbf{M} \to \infty$. We can now write $\mathbf{M}^{-1}\mathbf{M} = \mathbf{I}$ or in indicial notation,
    \begin{align*}
        \sum_{b=1}^{m}(\mathbf{M}^{-1})_{ab}(\mathbf{M})_{bc} = \boldsymbol{\delta}_{ac},
    \end{align*}
    where $\boldsymbol{\delta}_{ac}$ is the Kronecker delta function. In the above equation, $\forall k \in \mathcal{I}_{u,c}$, we will have
    \begin{align*}
        &\sum_{b=1}^{m}(\mathbf{M})_{kb}(\mathbf{M}^{-1})_{bc} = \boldsymbol{\delta}_{kc}\\
        \Rightarrow &\sum_{b=1}^{m}(\mathbf{M})_{kb}(\mathbf{M}^{-1})_{bc} \approx (\mathbf{M})_{kk}(\mathbf{M}^{-1})_{kc} =\boldsymbol{\delta}_{kc} \quad \text{as}\quad (\mathbf{M})_{kk} >> (\mathbf{M})_{bk}(b \neq k).
    \end{align*}
    When $k \neq c$, we have $\boldsymbol{\delta}_{kc} = 0$, so $(\mathbf{M}^{-1})_{kc} = 0 $, if $k \neq c$. Furthermore, $(\mathbf{M}^{-1})_{kk} = \frac{1}{(\mathbf{M})_{kk}} \to 0$ as $(\mathbf{M})_{kk} \to \infty$. So, the elements of $k$th row of $\mathbf{M}^{-1}$ are all zeros. Using this in eq. \eqref{eq:optimal_feedback}, we get that elements of $k$  rows of $\mathbf{K}_{t}^*$ are also zeros. Thus, there is no feedback when control is saturated.
\end{proof}
\begin{remark}
In practice, the rows of $\mathbf{K}_{t}$ corresponding to saturated controls do not become identically zero but rather converge to small, non-zero values. This is a direct consequence of using a logarithmic barrier, an interior-point method, which ensures the solution remains strictly within the feasible set. The control values only approach their limits as the barrier parameter tends to zero. While theoretically sound, this can lead to very large diagonal entries in the system matrix, potentially causing numerical ill-conditioning. To mitigate this, two common strategies are:
\begin{enumerate}
    \item Terminating the algorithm when the solution is within a desired tolerance of the boundary, leaving a small margin. In other words, terminate the algorithm at some significant values of barrier parameters.
    \item Incorporating a safety margin $\varepsilon > 0$ directly into the constraints, by enforcing control limits of $[\underline{u}_{j} + \varepsilon, \bar{u}_{j} - \varepsilon]$.
\end{enumerate}
This also implies that the feedback obtained at convergence of the algorithm, $\mathbf{K}_{t}^{*}$ will not influence the saturated control, if it is to be used to design a closed-loop around the optimal under uncertainty~\cite{naveed_feedback,parunandi2020d2c}.
\end{remark}

\begin{assumption}[Validity of Linearized Dynamics]
\label{ass:lq_validity}
From one iteration to the next, it is assumed that the evolution of the state perturbations is perfectly described by the linear model: $ \delta \mathbf{x}_{t+1} = \mathbf{F}_{\mathbf{x}}\delta \mathbf{x}_{t} + \mathbf{F}_{\mathbf{u}}\delta \mathbf{u}_{t}.$
\end{assumption}

\begin{theorem}[Feedback Gain Structure at State Boundaries]
\label{thm:gain_structure_state}
Under Assumption~\ref{ass:lq_validity}, consider the feedback law from eq.~\eqref{eq:u_update}. If the $i$-th component of the optimal state trajectory, $x_{t+1,i}$, is active at its constraint boundary defined in eq.~\eqref{eq:state_box_constraints_idx}, then the feedback gain matrix $\mathbf{K}_t$ is structured such that the control correction $\delta\mathbf{u}_t = \alpha \mathbf{k}_t + \mathbf{K}_t \delta\mathbf{x}_t$ has no net effect on the $i$-th component of the state deviation at the next time step, $\delta x_{t+1,i}$. Mathematically, this means the $i$-th component of the state deviation is governed purely by the open-loop dynamics: $ (\delta\mathbf{x}_{t+1})_i = (\mathbf{F}_{\mathbf{x}} \delta \mathbf{x}_t)_i.$
\end{theorem}
\begin{proof}
An active state constraint on $x_{t+1,i}$ is enforced by a penalty term in the cost function, causing the corresponding diagonal element of the state penalty Hessian, $(\boldsymbol{\omega}_{\mathbf{xx}})_{ii}$, to tend to infinity. This penalty propagates backward through the Riccati recursion.

$\mathbf{S}_{t+1}$, is computed during the backward pass. Its state-dependent component is dominated by the penalty term:
\begin{equation*}
    \mathbf{S}_{t+1} \approx \boldsymbol{\omega}_{\mathbf{xx}} + \dots
\end{equation*}
As $(\boldsymbol{\omega}_{\mathbf{xx}})_{ii} \to \infty$, the $i$-th diagonal element $(\mathbf{S}_{t+1})_{ii}$ also tends to infinity. For a symmetric matrix, this implies that its largest eigenvalue, $\eta_i$, tends to infinity, and the corresponding eigenvector, $\boldsymbol{v}_i$, converges to the $i$-th standard basis vector, $\mathbf{e}_i$. Furthermore, due to the orthogonality of eigenvectors of a symmetric matrix, all other eigenvectors $\boldsymbol{v}_k$ (for $k\neq i$) will be orthogonal to $\mathbf{e}_i$, meaning their $i$-th component will be zero.

At time step $t$, the Hessian of the action-value function, $\mathbf{Q}_{\mathbf{uu}}$, is given by:
\begin{equation*}
    \mathbf{Q}_{\mathbf{uu}} = C_{\mathbf{uu}} + \boldsymbol{\omega}_{\mathbf{uu}} + \mathbf{F}_{\mathbf{u}}^{\top}\mathbf{S}_{t+1}\mathbf{F}_{\mathbf{u}}.
\end{equation*}
Let $\mathbf{P} = \mathbf{F}_{\mathbf{u}}^{\top}\mathbf{S}_{t+1}\mathbf{F}_{\mathbf{u}}$. Using the spectral decomposition $\mathbf{S}_{t+1} = \sum_{k=1}^{n} \eta_k \boldsymbol{v}_k \boldsymbol{v}_k^{\top}$, the $(a,b)$-th element of $\mathbf{P}$ is:
\begin{equation*}
    (\mathbf{P})_{ab} = (\mathbf{F_u}^a)^{\top} \mathbf{S}_{t+1} \mathbf{F_u}^b = \sum_{k=1}^{n} \eta_k \left( \boldsymbol{v}_k^{\top}(\mathbf{F_u}^a)  \right) \left( (\mathbf{F_u}^b)^{\top} \boldsymbol{v}_k \right),
\end{equation*}
where $\mathbf{F_u}^a$ and $\mathbf{F_u}^b$ are the $a$-th and $b$-th column of $\mathbf{F}_{\mathbf{u}}$. As $\eta_i \to \infty$ and $\boldsymbol{v}_i \to \mathbf{e}_i$, this sum is dominated by the $i$-th term:
\begin{equation*}
    (\mathbf{P})_{ab} \approx \eta_i \left( \mathbf{e}_i^{\top}(\mathbf{F_u}^a)  \right) \left( (\mathbf{F_u}^b)^{\top} \mathbf{e}_i \right) = \eta_i (\mathbf{F}_{\mathbf{u}})_{ia} (\mathbf{F}_{\mathbf{u}})_{ib}.
\end{equation*}
This shows that if $a$-th and $b$-th component of control $\mathbf{u}_t$ influences state component $(\mathbf{x}_t)_i$ (i.e., $(\mathbf{F}_{\mathbf{u}})_{ia} \neq 0$ and $(\mathbf{F}_{\mathbf{u}})_{ib} \neq 0$), then the elements $(\mathbf{P})_{ij}\to \infty~ \forall i,j \in a,b$,  and consequently $(\mathbf{Q}_{\mathbf{uu}})_{ij} \to \infty~\forall i,j \in a,b$.

The feedback gain is $\mathbf{K}_t^* = -(Q_{\mathbf{uu}})^{-1}Q_{\mathbf{ux}}$. Let $\mathcal{I}_i = \{a \mid (\mathbf{F}_{\mathbf{u}})_{ia} \neq 0\}$ be the set of indices of controls that influence the $i$-th state. For any $a \in \mathcal{I}_i$, the corresponding diagonal term $(\mathbf{Q}_{\mathbf{uu}})_{aa}$ tends to infinity. Using the result from theorem \ref{thm:gain_structure_control}, the rows of a matrix inverse corresponding to an infinite diagonal term tend to zero. Therefore, the $a$-th row of $(\mathbf{Q}_{\mathbf{uu}})^{-1}$ tends to the zero vector for all $a \in \mathcal{I}_i$. Consequently, the $a$-th row of $\mathbf{K}_t^*$ must also be zero for all $a \in \mathcal{I}_i$. This is true even when multiple elements in a row goes to $\infty$.

The control correction is $\delta\mathbf{u}_t = -\mathbf{K}_t^*\delta\mathbf{x}_t$. The $a$-th component $(\delta\mathbf{u}_t)_a$ is zero if $a \in \mathcal{I}_i$. Now consider the $i$-th component of the control's effect on the next state's deviation, $(\mathbf{F}_{\mathbf{u}}\delta\mathbf{u}_t)_i$:
\begin{equation*}
    (\mathbf{F}_{\mathbf{u}}\delta\mathbf{u}_t)_i = \sum_{a=1}^{m} (\mathbf{F}_{\mathbf{u}})_{ia} (\delta\mathbf{u}_t)_a.
\end{equation*}
This sum is identically zero. Each term is zero because either $(\mathbf{F}_{\mathbf{u}})_{ia} = 0$ (if $a \notin \mathcal{I}_i$) or $(\delta\mathbf{u}_t)_a = 0$ (if $a \in \mathcal{I}_i$).
Therefore, the $i$-th component of the full state deviation simplifies to:
\begin{equation*}
    (\delta\mathbf{x}_{t+1})_i = (\mathbf{F}_{\mathbf{x}} \delta \mathbf{x}_t)_i + (\mathbf{F}_{\mathbf{u}}\delta\mathbf{u}_t)_i = (\mathbf{F}_{\mathbf{x}} \delta \mathbf{x}_t)_i.
\end{equation*}
This demonstrates that the feedback controller effectively ``disconnects'' its influence on the saturated state component, ensuring the constraint is not violated by the feedback control action.
\end{proof}

From theorem~\ref{thm:gain_structure_control} and theorem~\ref{thm:gain_structure_state}, one can say that the feedback gains obtained from Box-iLQR are ``constraint-aware''.

\begin{remark}
The feedback law obtained at convergence, $\mathbf{K}_{t}^{*}$, can be used to design a closed-loop controller around the optimal under uncertainty. A practical strategy is to keep some finite value of the barrier parameter $\boldsymbol{\mu}$ and not let it diminish to very small values. This ensures the nominal trajectory remains strictly within the bounds $[\underline{x}_{i} + \varepsilon, \bar{x}_{i} - \varepsilon]$. Under stochastic setting, when the next state is effected by the previous, this margin helps maintain constraint satisfaction. Moreover, the feedback law in eq.~\eqref{eq:feedback_law} is aware of the constraint coming ahead at previous time steps as the constraint information is propagated through backward pass in eq.~\eqref{eq:v_update} and eq.~\eqref{eq:S_update}. 
\end{remark}

\section{Results and Discussion}\label{sec:results}
We validate our theoretical results with numerical simulations on three benchmark swing-up tasks: the pendulum, cart-pole, and acrobot. All simulations involve systems with control and/or state constraints. Each problem employs a quadratic cost function and is discretized using the method from Lemma~\ref{lemma:dis_con_eq}. The simulation parameters for each system are listed in Table~\ref{tab:system_parameters}. Table~\ref{tab:hyperparameters} summarizes the hyperparameters for the Box-iLQR algorithm.

\begin{table}[H]
    \centering
    \begin{tabular}{@{}l|c|c|c|c|c|c|c l@{}}
    \toprule
    System & $Q$ & $R$ & $Q_{f}$ & $t_{f}$  & $\Delta t$  & State Constraints & Control Constraints\\
    \midrule
    Pendulum & $3\mathbf{I}_{2}$ & 3 & $30\mathbf{I}_{2}$ & 5 & 0.01 & --- & $-1\leq u \leq 1$ \\
    Cart-pole & $10\mathbf{I}_{4}$ & 10 & $10^4 \mathbf{I}_{4}$ & 10 & 0.01 & $-0.2 \leq x_1 \leq 0.2$ & $-2 \leq u \leq 2$ \\
    Acrobot & $500\mathbf{I}_{4}$ & 10 & $5 \times 10^4 \mathbf{I}_{4}$ & 10 & 0.01 & --- & $-5 \leq u \leq 5$\\
    \bottomrule
    \end{tabular}
    \caption{Parameters for the optimal control problems. Here, $Q$ and $R$ are the weighting matrices for the quadratic state and control stage costs, respectively, while $Q_f$ is the terminal state cost weighting matrix.}
    \label{tab:system_parameters}
\end{table}

\subsection{Pendulum}

\begin{figure}
    \centering
    \begin{subfigure}{0.32\textwidth}
        \includegraphics[width=\textwidth]{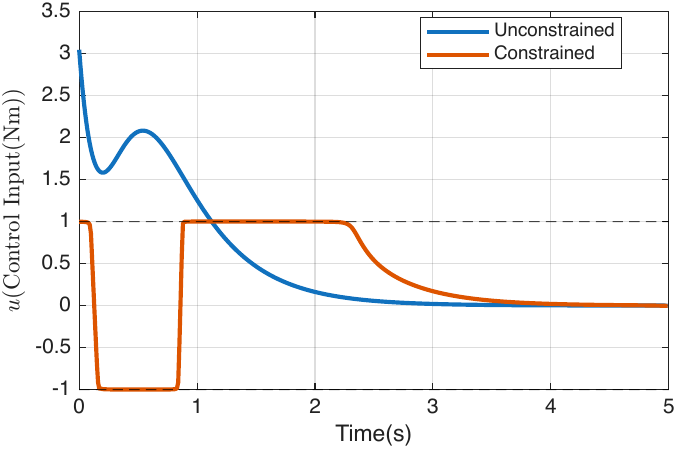}
            \caption{Control input $u(t)$.}
            \label{fig:pendulum_control_comp}
        \end{subfigure}
        \hfill
        \begin{subfigure}{0.32\textwidth}
            \includegraphics[width=\textwidth]{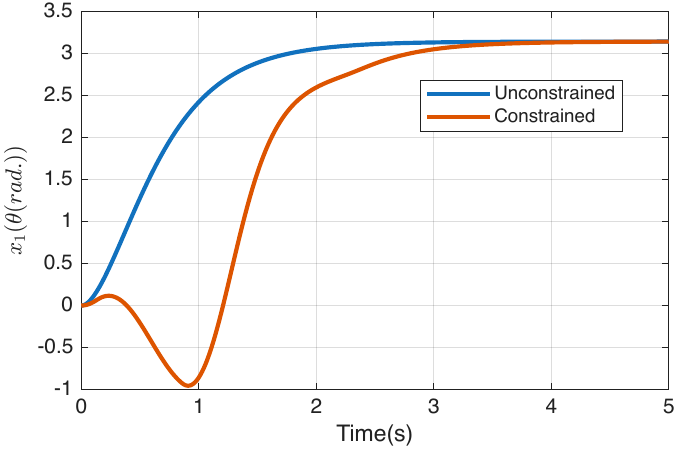}
            \caption{$x_{1}(t)$.}
            \label{fig:pendulum_state1_comp}
        \end{subfigure}
        \hfill
        \begin{subfigure}{0.32\textwidth}
            \includegraphics[width=\textwidth]{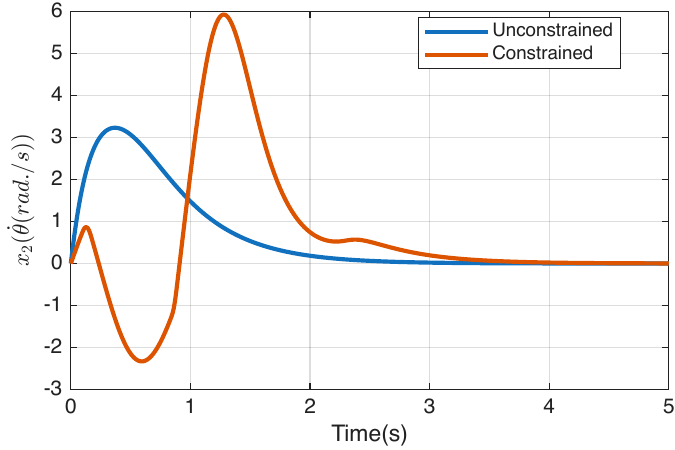}
            \caption{$x_{2}(t)$.}
            \label{fig:pendulum_state2_comp}
        \end{subfigure}
        \caption{Comparison of unconstrained and constrained control for the pendulum swing-up task. The plots show the resulting (a) control input $u(t)$, (b) pendulum angle $x_1(t) = \theta(t)$, and (c) angular velocity $x_2(t) = \dot{\theta}(t)$ over time. The unconstrained case is shown in blue and the constrained case in red.}
\label{fig:pendulum_comp_main}
\end{figure}
Figure~\ref{fig:pendulum_comp_main} presents a comparison of the state and control trajectories for the pendulum swing-up task, contrasting the unconstrained controller with the constrained one. The physical system parameters, detailed in Table~\ref{tab:system_parameters} are same for both cases, except that the control torque for the constrained case is limited to $\pm1$~N$\cdot$m.

The unconstrained system, depicted by the blue curves, employs a direct swing-up strategy. As illustrated in Fig.~\ref{fig:pendulum_control_comp}, the controller applies a continuous, high-magnitude positive torque. This action results in a monotonic increase in the pendulum's angle, $x_1(t)$, which smoothly converges to the upright position at $\pi$~rad (Fig.~\ref{fig:pendulum_state1_comp}). Concurrently, the angular velocity, $x_2(t)$, converges to zero, indicating successful stabilization at the target state (Fig.~\ref{fig:pendulum_state2_comp}). This direct, single-direction motion is further corroborated by the trajectory visualization in Fig.~\ref{fig:pendulum_traj_comp} (light blue), which shows the pendulum reaching the upright configuration by $t=2.22$~s.

In stark contrast, the constrained controller's limited authority prevents a direct swing-up. Instead, the system must execute an energy-pumping maneuver to accumulate sufficient momentum. This strategy begins with a small initial swing in the positive $\theta$ direction, followed by a much larger swing in the opposite direction to build kinetic energy. This behavior is evident in the state trajectory (Fig.~\ref{fig:pendulum_state1_comp}, red curve) and is driven by the control input saturating at its lower bound (Fig.~\ref{fig:pendulum_control_comp}). The snapshots at $t=0.56$~s and $t=1.11$~s in Fig.~\ref{fig:pendulum_traj_comp} (red) visually capture this essential counter-swing. Once sufficient energy is gained, the controller again saturates at its positive limit to drive the pendulum towards the upright position, where it is subsequently stabilized as the control input diminishes.

This analysis highlights a fundamental difference in control strategies dictated by actuator limitations. While the unconstrained system can rely on brute force for a direct maneuver, the constrained system must strategically exploit the natural dynamics of the pendulum to achieve the same swing-up objective.
\begin{figure}[!htbp]
    \centering
    \includegraphics[width=1\linewidth]{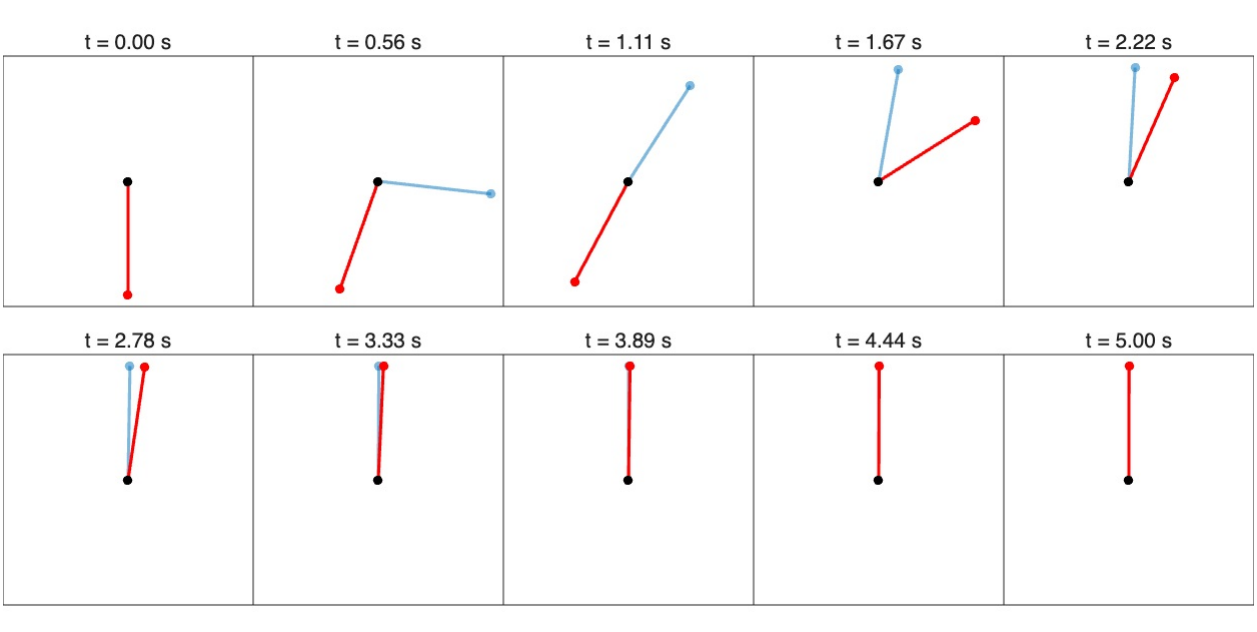}
    \caption{A sequence of snapshots illustrating the pendulum swing-up trajectory. The motion resulting from the unconstrained controller is shown in light blue, while the trajectory under the constrained controller is shown in red.}
\label{fig:pendulum_traj_comp}
\end{figure}

\subsection{Cart-pole}
As detailed in Table~\ref{tab:system_parameters}, the cart-pole system is analyzed under both control and state constraints, defined by $|u| \leq 2$~N and $|x_1| \leq 0.2$~m, respectively. We evaluate the performance of our method against an unconstrained baseline in three distinct scenarios: (a) Cart-pole with control constraints only, (b) Cart-pole with state constraints only, and (c) Cart-pole with both state and control constraints.
\begin{table}
    \centering
    \begin{tabular}{@{}l|c|c|c|c|c|c l@{}}
        \toprule
        \textbf{Hyperparameters} & $\boldsymbol{\mu}_{0}$ & $\boldsymbol{\sigma}_{0}$ & $\epsilon_{\text{barrier}}$ & $\mathbf{r}_{\mu}$ & $\mathbf{r}_{\sigma}$ & $\beta_{r}$ \\
        \midrule
         \textbf{Value} & $1e8\mathbf{I}$ & $1e8\mathbf{I}$ & 0.01 & 0.5 & 0.5 & (1/0.95)\\
         \bottomrule
    \end{tabular}
    \caption{Hyperparameters used for Box-iLQR.}
    \label{tab:hyperparameters}
\end{table}
\begin{figure}[!ht]
    \centering
    \begin{subfigure}{0.32\textwidth}
        \includegraphics[width=\textwidth]{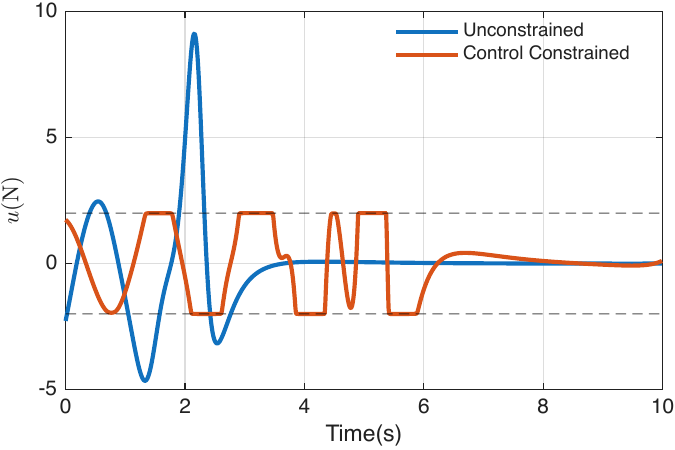}
            \caption{Control constrained.}
            \label{fig:cartpole_control_comp1}
    \end{subfigure}
    \hfill
    \begin{subfigure}{0.32\textwidth}
        \includegraphics[width=\textwidth]{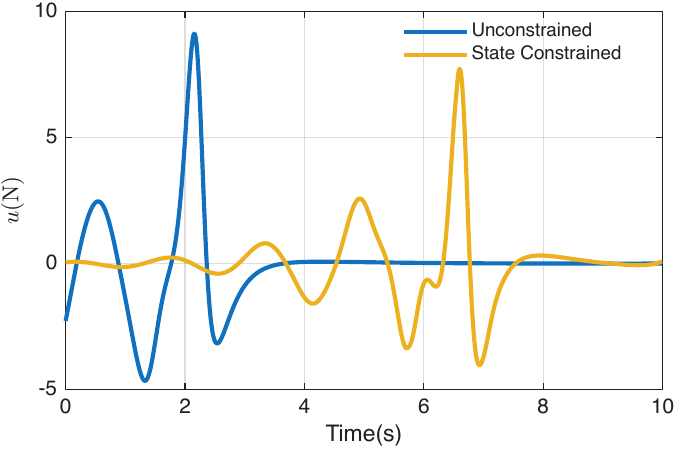}
            \caption{State constrained.}
            \label{fig:cartpole_control_comp2}
    \end{subfigure}
    \hfill
    \begin{subfigure}{0.32\textwidth}
        \includegraphics[width=\textwidth]{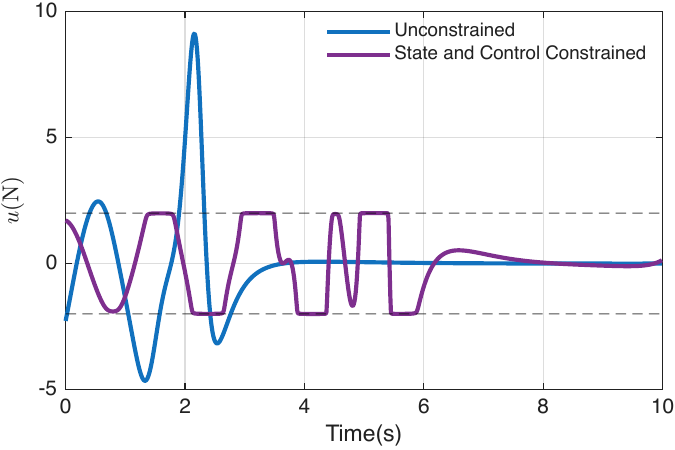}
            \caption{Both constrained.}
            \label{fig:cartpole_control_comp3}
    \end{subfigure}
     \caption{Comparison of control inputs $u(t)$ for the cart-pole swing-up. Each subplot contrasts the unconstrained control trajectory (blue) with different constrained scenarios: (\subref{fig:cartpole_control_comp1}) control constrained, (\subref{fig:cartpole_control_comp2}) state constrained, and (\subref{fig:cartpole_control_comp3}) a combination of both. The control constraint boundaries, $-2 \leq u \leq 2$, are shown with dashed black lines in (\subref{fig:cartpole_control_comp1}) and (\subref{fig:cartpole_control_comp3}).}
    \label{fig:cartpole_control_comp_main}
\end{figure}

\begin{figure}[!ht]
    \begin{minipage}{0.48\textwidth}
        \begin{subfigure}{\textwidth}
        \includegraphics[width=\textwidth]{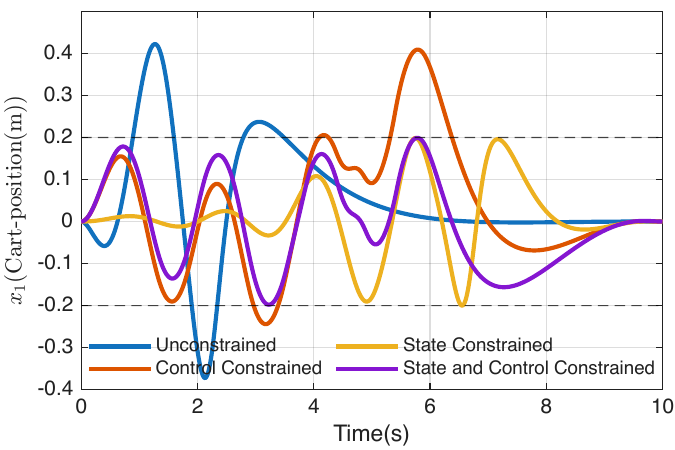}
        \caption{$x_{1}(t)$(Cart-position).}
        \label{fig:cartpole_state1}
        \end{subfigure}
        \begin{subfigure}{\textwidth}
        \includegraphics[width=\textwidth]{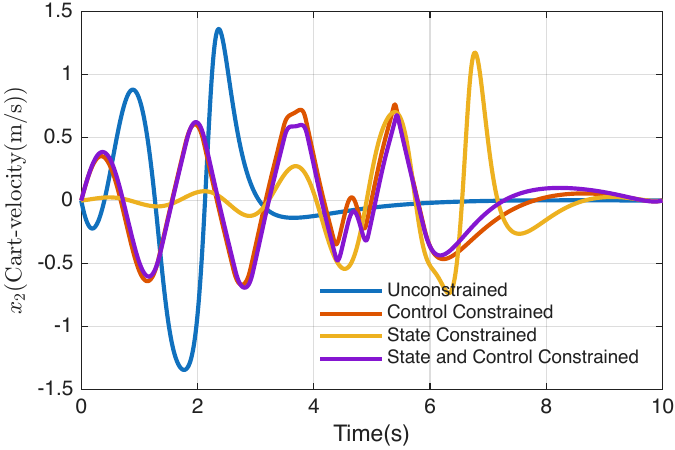}
        \caption{$x_{2}(t)$(Cart-velocity).}
        \label{fig:cartpole_state2}
        \end{subfigure}
    \end{minipage}
    \hfill
    \begin{minipage}{0.48\textwidth}
        \begin{subfigure}{\textwidth}
        \includegraphics[width=\textwidth]{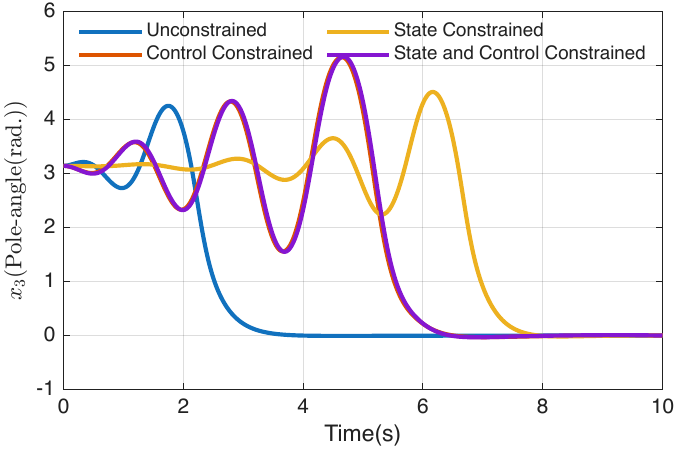}
        \caption{$x_{3}(t)$(Pole-angle).}
        \label{fig:cartpole_state3}
        \end{subfigure}
        \begin{subfigure}{\textwidth}
        \includegraphics[width=\textwidth]{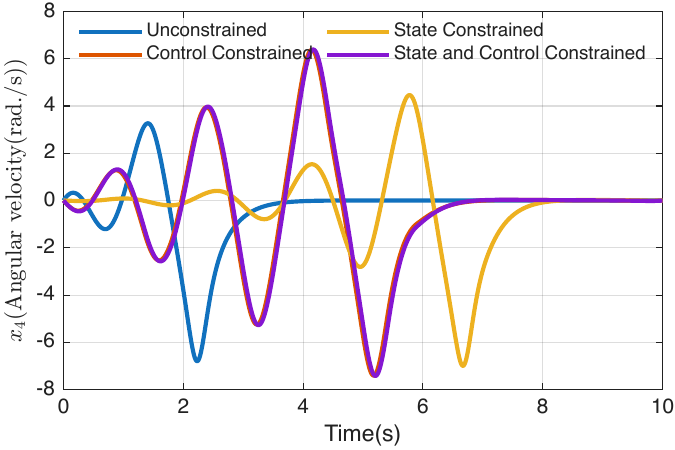}
        \caption{$x_{4}(t)$(Pole's angular velocity).}
        \label{fig:cartpole_state4}
        \end{subfigure}
    \end{minipage}
    \caption{A comparative analysis of state evolution for the cart-pole swing-up for the unconstrained and the three constrained conditions.}
    \label{fig:cartpole_state_comp}
\end{figure}

The different control strategies adopted by the system are shown in Figure~\ref{fig:cartpole_control_comp_main}. The unconstrained controller applies a large control input peaking near $t=2s$ (approx. 10~N) to achieve a rapid swing-up. In contrast, when control is constrained, the system must adopt an energy-pumping strategy. This is evident in the control profiles for both the control-constrained (Fig.~\ref{fig:cartpole_control_comp1}, red) and the fully constrained (Fig.~\ref{fig:cartpole_control_comp3}, purple) cases, which exhibit more oscillations between the saturation bounds of $\pm 2$~N. The state-constrained case (Fig.~\ref{fig:cartpole_control_comp2}, yellow) presents a third, more cautious strategy: the controller initially applies smaller inputs, ramping up only later in the maneuver to prevent the cart from violating its position constraints.

The resulting state trajectories, shown in Figure~\ref{fig:cartpole_state_comp}, reveal the consequences of these control strategies. As expected, the cart position ($x_1$) in Figure~\ref{fig:cartpole_state1} confirms that both the state-constrained and the fully constrained scenarios respect the position boundaries at $\pm 0.2$~m, shown in yellow and purple, respectively. A key insight arises from comparing the control-constrained (red) and fully constrained (purple) cases. While their pole dynamics ($x_3$ and $x_4$) are nearly indistinguishable (Fig.~\ref{fig:cartpole_state3} and \ref{fig:cartpole_state4}), a subtle but critical difference emerges in the cart's velocity profile ($x_2$, Fig.~\ref{fig:cartpole_state2}). This minor adjustment in the velocity trajectory is precisely what enables the fully constrained system to satisfy the position constraint, whereas the purely control-constrained system fails to do so.

\begin{figure}[!ht]
    \centering
    \includegraphics[width=0.9\linewidth]{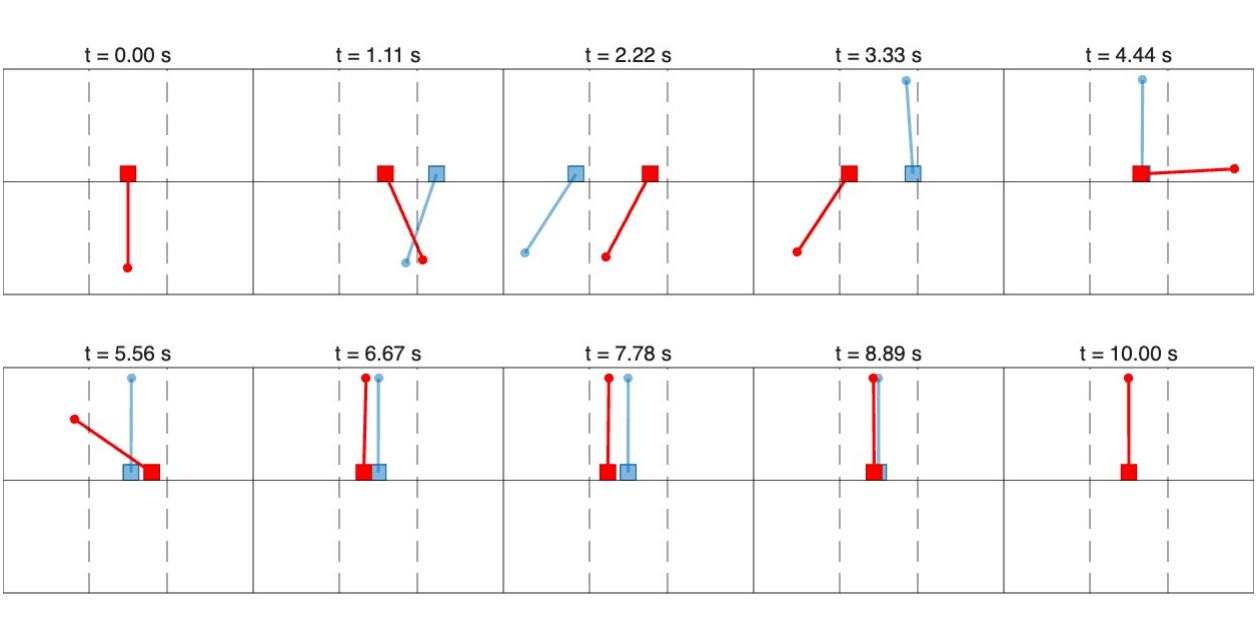}
    \caption{Time-lapse visualization of the cart-pole swing-up task, comparing the unconstrained trajectory (light blue) against the trajectory with both state and control constraints (red). The dashed black lines indicate the position constraints on the cart, $-0.2 \leq x_{1} \leq 0.2$.}
    \label{fig:cartpole_traj_comp}
\end{figure}

Figure~\ref{fig:cartpole_traj_comp} provides a time-lapse visualization that further illustrates these findings by comparing the unconstrained and fully constrained cases. The unconstrained system (light blue) is shown to violate the state boundaries at multiple instances (e.g., at $t=1.11\text{ s, } 2.22\text{ s, and } 3.33\text{ s}$) while achieving a rapid swing-up, stabilizing around $t=5.56\text{ s}$. Conversely, the constrained system (red) must perform a more complex, oscillatory maneuver to build momentum, and is still in motion at the end of the observed time frame, demonstrating the performance under constrained scenario.

\subsection{Acrobot}
\begin{figure}[!htbp]
    \centering
    \begin{subfigure}{0.32\textwidth}
        \includegraphics[width=\textwidth]{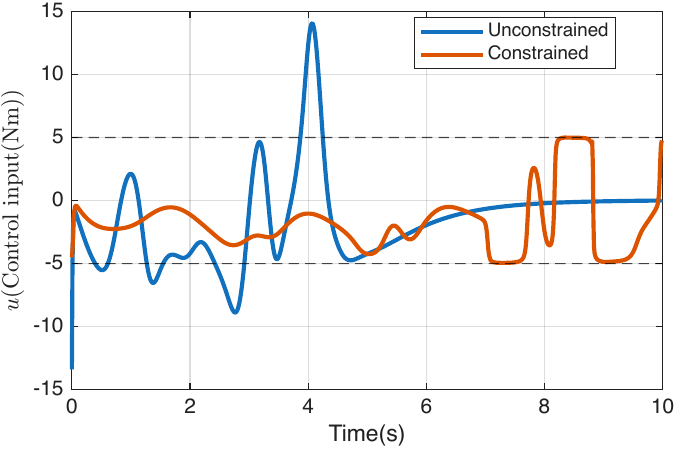}
        \caption{$u(t)$.}
        \label{fig:acrobot_control_comp}
    \end{subfigure}
    \begin{subfigure}{0.32\textwidth}
        \includegraphics[width=\textwidth]{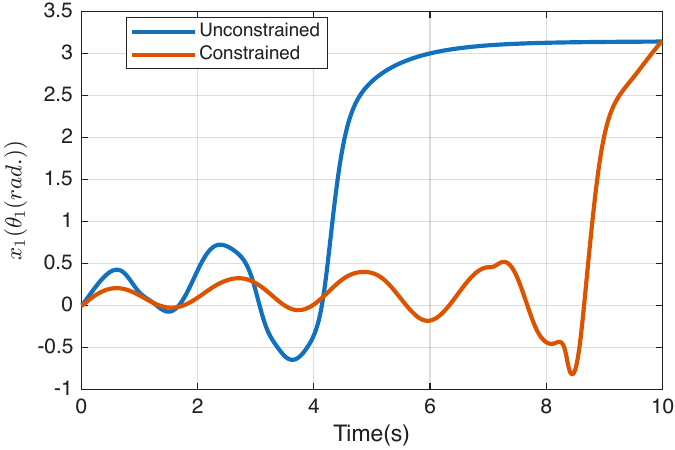}
        \caption{$x_{1}(t)$.}
        \label{fig:acrobot_state1_comp}
    \end{subfigure}
     \begin{subfigure}{0.32\textwidth}
        \includegraphics[width=\textwidth]{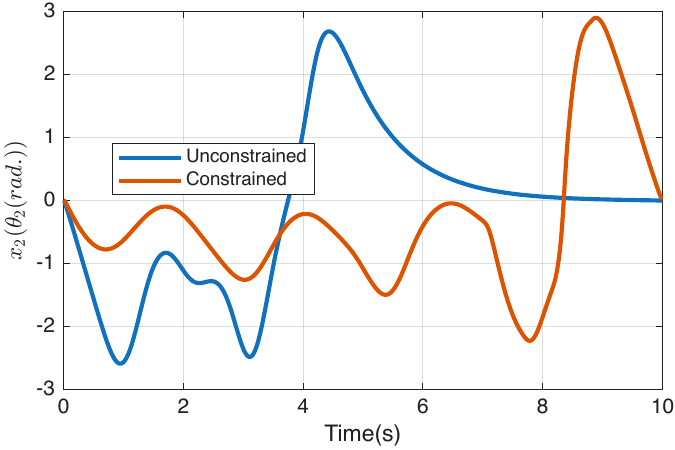}
        \caption{$x_{2}(t)$.}
        \label{fig:acrobot_state2_comp}
    \end{subfigure}
    \begin{subfigure}{0.32\textwidth}
        \includegraphics[width=\textwidth]{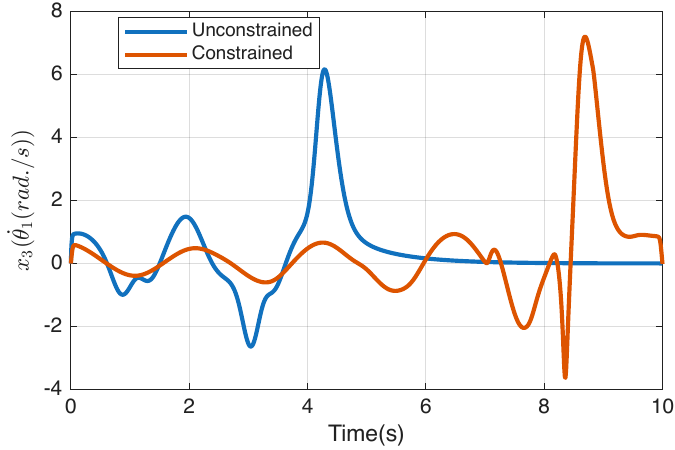}
        \caption{$x_{3}(t)$.}
        \label{fig:acrobot_state3_comp}
    \end{subfigure}
    \begin{subfigure}{0.32\textwidth}
        \includegraphics[width=\textwidth]{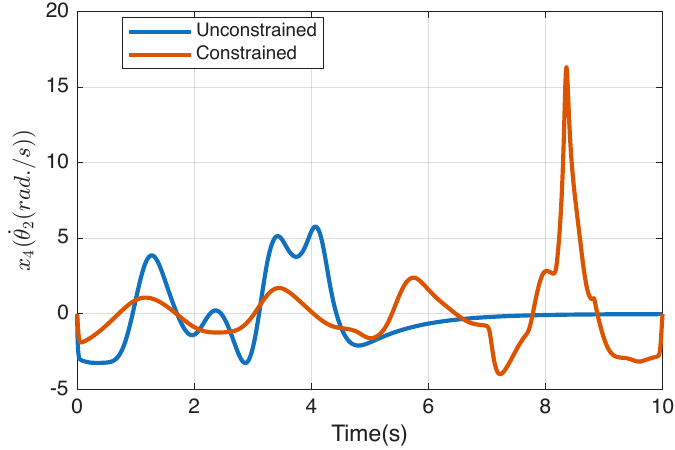}
        \caption{$x_{4}(t)$.}
        \label{fig:acrobot_state4_comp}
    \end{subfigure}
    \caption{State and control trajectories for the acrobot swing-up task, comparing the unconstrained (blue) and control-constrained (red) cases. The subplots show the time evolution of the angles $\theta_1$ (\subref{fig:acrobot_state1_comp}) and $\theta_2$ (\subref{fig:acrobot_state2_comp}), their corresponding angular velocities $\dot{\theta}_1$ (\subref{fig:acrobot_state3_comp}) and $\dot{\theta}_2$ (\subref{fig:acrobot_state4_comp}), and the applied control torque $u(t)$ (\subref{fig:acrobot_control_comp}). The control input is bounded by the limits $u = \pm 5$, indicated by the dashed lines in (\subref{fig:acrobot_control_comp}).}
\label{fig:acrobot_comp_main}
\end{figure}
The state and control trajectories for the acrobot are presented in Figure~\ref{fig:acrobot_comp_main}, comparing the unconstrained (blue) and control-constrained (red) swing-up strategies. The unconstrained controller employs a high-magnitude, aggressive strategy, with inputs ranging from approximately -13~N$\cdot$m to +15~N$\cdot$m (Fig.~\ref{fig:acrobot_control_comp}). This direct approach enables a rapid swing-up, achieving stabilization at the upright position around $t=6-7$~s, which is confirmed by the control input and all state trajectories (Figs.~\ref{fig:acrobot_control_comp}-\ref{fig:acrobot_state4_comp}) converging to zero.

In contrast, the control-constrained system adopts a more patient, energy-pumping maneuver, similar to the strategies observed for the pendulum and cart-pole. The system builds kinetic energy over a longer duration before executing a final, powerful swing to reach the upright configuration. This strategy is reflected in the state trajectories, which exhibit sharp, rapid changes late in the maneuver as the final swing is performed.

The trajectory snapshots in Figure~\ref{fig:acrobot_traj_comp} provide a visual corroboration of these distinct strategies. The unconstrained system (light blue) initiates its primary swing around $t=3.39$~s, reaches the upright position by $t=5.60$~s, and appears fully stabilized by $t=7.80$~s. Conversely, the constrained acrobot (red) only begins its final, decisive swing near $t=7.80$~s and does not achieve the upright configuration until the final snapshot at $t=10$~s, clearly illustrating the performance trade-off inherent in operating under actuation limits.

\begin{figure}[!htbp]
    \centering
    \includegraphics[width=0.9\linewidth]{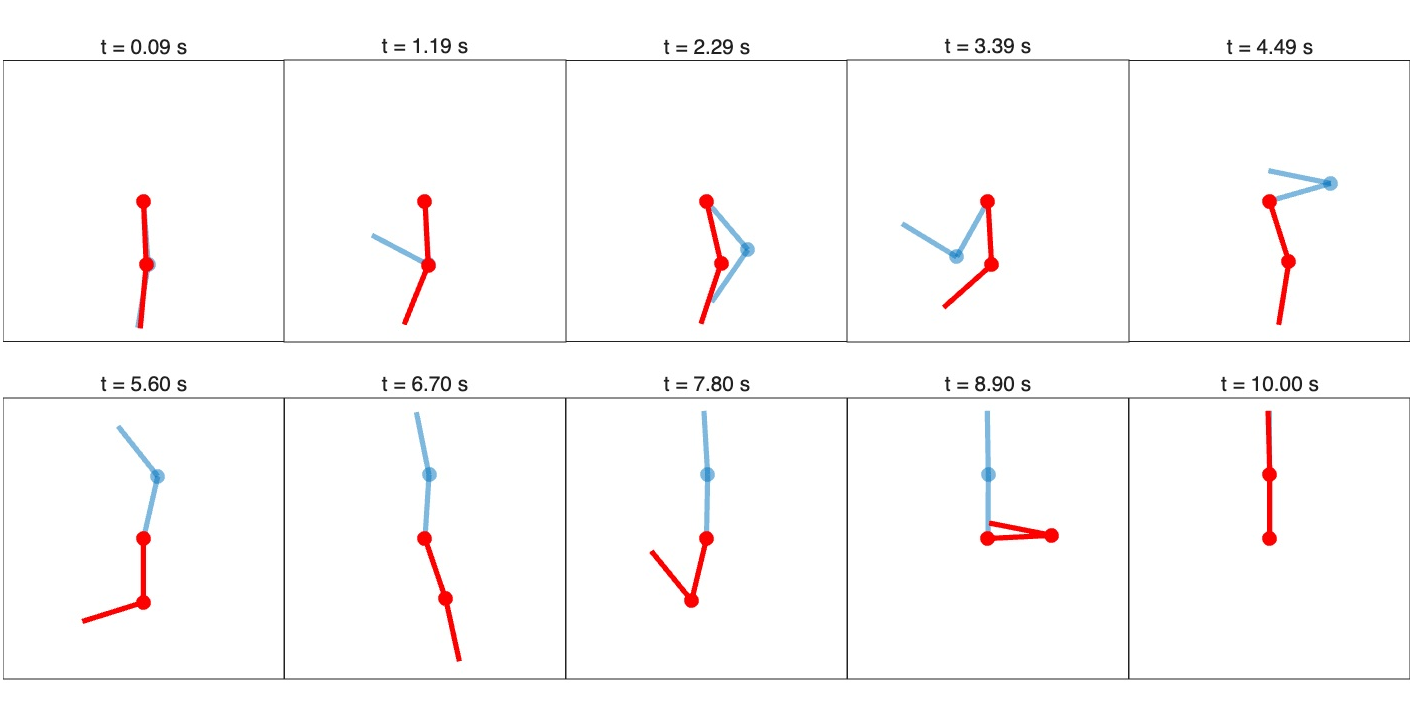}
    \caption{Comparison of Acrobot swing-up trajectories, contrasting the unconstrained (light blue) case with the control-constrained (red) case.}
\label{fig:acrobot_traj_comp}
\end{figure}

\subsection{Analysis of the Feedback Terms}

As claimed in Subsection~\ref{subsec:constraint-aware}, the feedback law generated by Box-iLQR is "constraint-aware," meaning the feedback gains diminish as the control input approaches its saturation boundary. This section provides numerical validation for this property using the pendulum, cart-pole, and acrobot swing-up tasks. All simulations were run to convergence with the barrier parameters set to $\mu = \sigma = 0.01$ (see Table~\ref{tab:system_parameters}). As is characteristic of interior-point methods with a non-zero barrier, the control input remains strictly within the feasible set, approaching but never perfectly reaching the saturation boundary.

Across all three systems, the results align with the theory presented in Theorem~\ref{thm:gain_structure_control}. In the corresponding figures, regions where the control input is near saturation are highlighted in gray. For the pendulum (Figure~\ref{fig:pendulum_feedback}), the feedback gains in Figures~\ref{fig:pendulum_feedback1} and~\ref{fig:pendulum_feedback2} are observed to approach zero within these gray regions. A similar behavior is observed for the cart-pole system. As the control input nears its bounds (Figure~\ref{fig:cartpole_control_sat}), all four components of the feedback gain, $K_1$ through $K_4$, diminish significantly, as shown in Figures~\ref{fig:cartpole_feedback1}--\ref{fig:cartpole_feedback4}.

The acrobot system also demonstrates this trend during its near-saturation phases (Figure~\ref{fig:acrobot_control_sat}). While most feedback components are visibly near zero (Figures~\ref{fig:acrobot_feedback1}--\ref{fig:acrobot_feedback4}), the fourth gain, $K_4$ (Figure~\ref{fig:acrobot_feedback4}), appears non-zero due to the plot's y-axis scaling. However, its value is still substantially smaller than its peak magnitude. This is consistent with the theoretical prediction that the gains become small, but not identically zero, for a finite barrier parameter $\epsilon_{\text{barrier}}$.

Furthermore, a second noteworthy phenomenon is consistently observed across all simulations: sharp spikes in the feedback gains immediately preceding and following the periods of near-saturation. This behavior suggests a compensatory mechanism. As the controller anticipates entering or exiting a phase where its corrective authority is limited by the constraint, it prepares to apply larger feedback corrections.

Finally, examining the feedback profiles for the cart-pole (Figs.~\ref{fig:cartpole_feedback1}--\ref{fig:cartpole_feedback4}) after the final saturation phase (approximately $t>6$~s), we observe that the gains remain smooth before spiking near the end of the trajectory. This terminal spike is a distinct phenomenon, driven by the high terminal cost specified in Table~\ref{tab:system_parameters}, as the controller prepares to make large corrections to ensure the final state constraints are precisely met.

\begin{figure}[!htbp]
    \centering
    
    \begin{subfigure}{\textwidth}
            \centering
            \includegraphics[width=0.6\linewidth]{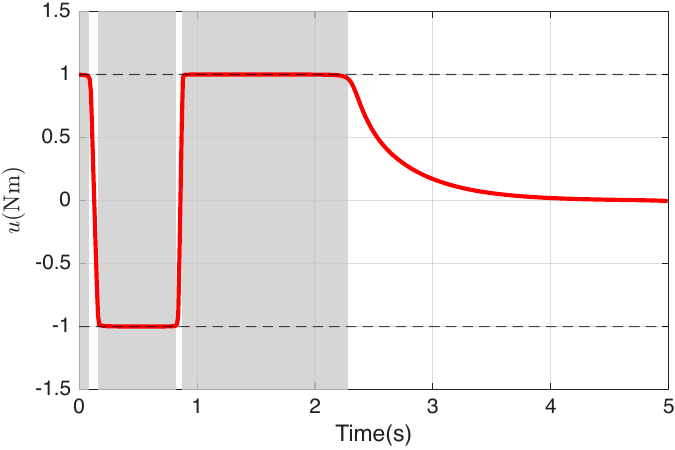}
    \caption{Control input - $u(t)$.}
    \label{fig:pendulum_control_sat}
    \end{subfigure}
    \hfill
    \begin{subfigure}{0.48\textwidth}
            \includegraphics[width=1\linewidth]{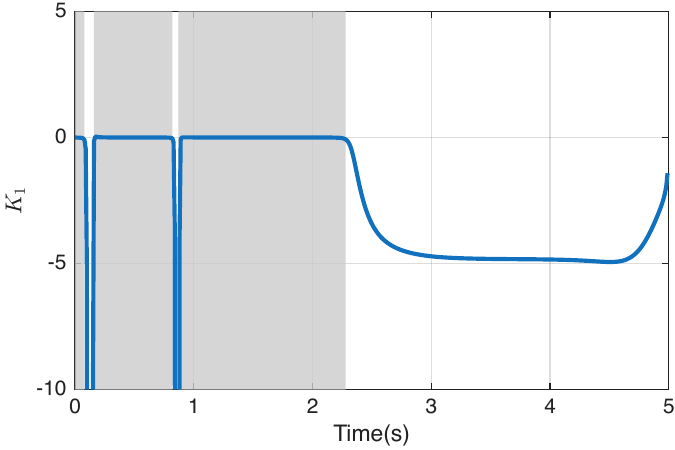}
    \caption{Feedback gain - $K_{1}(t)$.}
    \label{fig:pendulum_feedback1}
    \end{subfigure}
    \begin{subfigure}{0.48\textwidth}
            \includegraphics[width=1\linewidth]{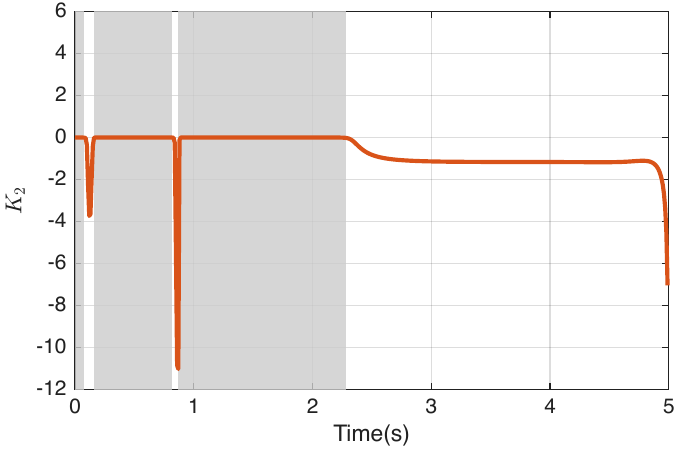}
    \caption{Feedback gain - $K_{2}(t)$.}
    \label{fig:pendulum_feedback2}
    \end{subfigure}
    \caption{Plots of control and components of feedback, $\mathbf{K}_t^{*}$, for the pendulum swing-up. The shaded region in the plots show the area where control nears saturation and the feedback terms tend to zero.}
    \label{fig:pendulum_feedback}
\end{figure}

\begin{figure}[!htbp]
    \centering
    \begin{subfigure}{0.6\textwidth}
            \includegraphics[width=1\linewidth]{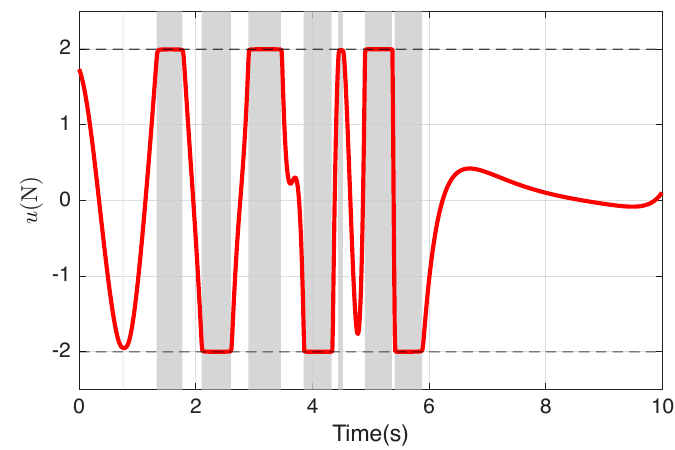}
    \caption{Control input - $u(t)$.}
    \label{fig:cartpole_control_sat}
    \end{subfigure}
    
    \begin{subfigure}{0.48\textwidth}
            \includegraphics[width=1\linewidth]{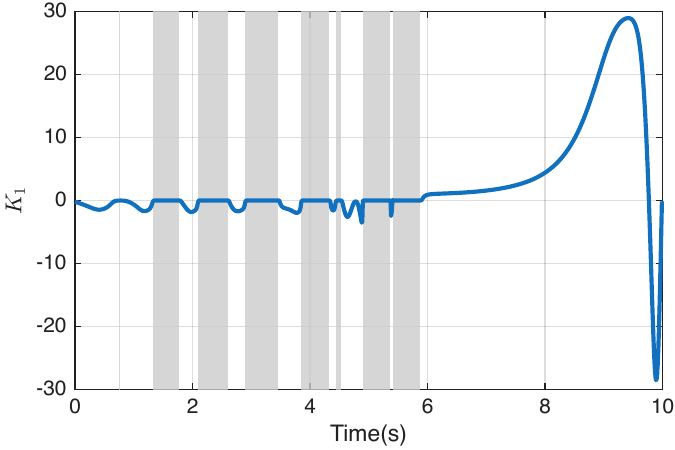}
    \caption{$K_{1}(t)$.}
    \label{fig:cartpole_feedback1}
    \end{subfigure}
    \hfill
    \begin{subfigure}{0.48\textwidth}
            \includegraphics[width=1\linewidth]{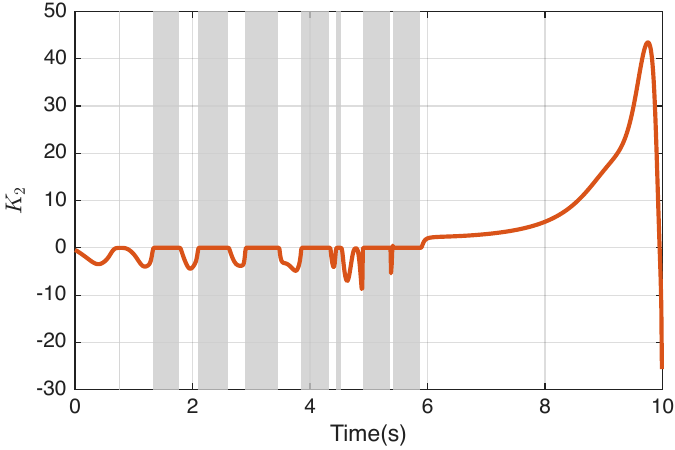}
    \caption{$K_{2}(t)$.}
    \label{fig:cartpole_feedback2}
    \end{subfigure}

    \begin{subfigure}{0.48\textwidth}
            \includegraphics[width=1\linewidth]{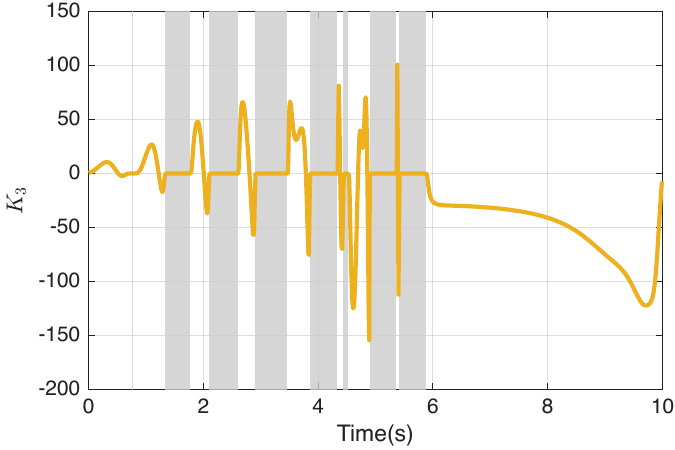}
    \caption{$K_{3}(t)$.}
    \label{fig:cartpole_feedback3}
    \end{subfigure}
    \hfill    
    \begin{subfigure}{0.48\textwidth}
            \includegraphics[width=1\linewidth]{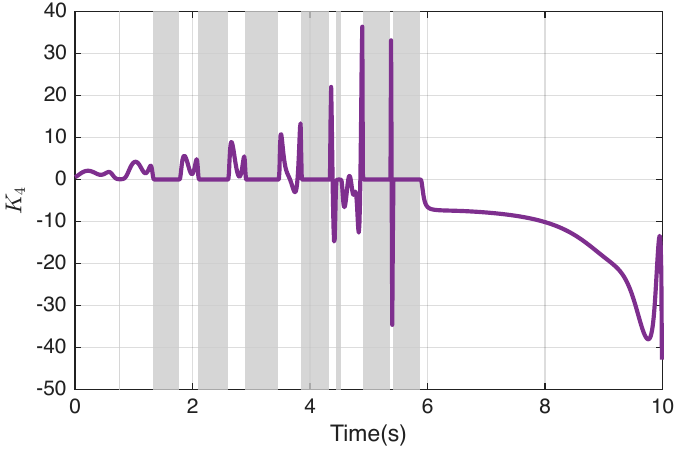}
    \caption{$K_{4}(t)$.}
    \label{fig:cartpole_feedback4}
    \end{subfigure}
    \caption{Plots of control input and components of feedback, $\mathbf{K}_t^{*}$, for the cart-pole swing-up. The shaded region in the plots show the area where control nears saturation and the feedback terms tend to zero.}
    \label{fig:cartpole_feedback}
\end{figure}

\begin{figure}[!htbp]
    \centering
    \begin{subfigure}{0.6\textwidth}
            \includegraphics[width=1\linewidth]{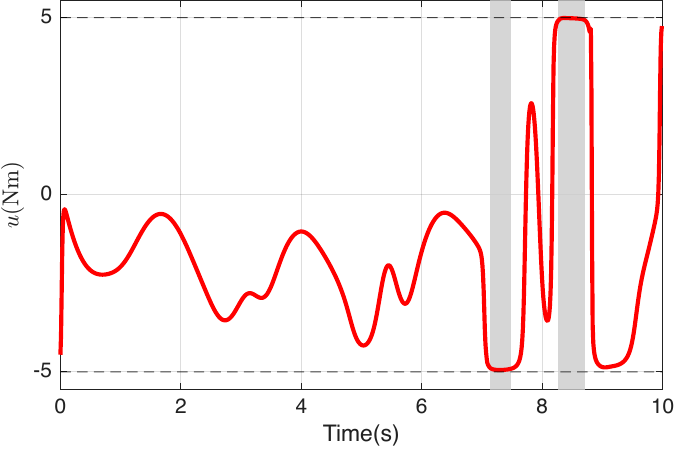}
    \caption{$u(t)$.}
    \label{fig:acrobot_control_sat}
    \end{subfigure}
    \begin{subfigure}{0.48\textwidth}
            \includegraphics[width=1\linewidth]{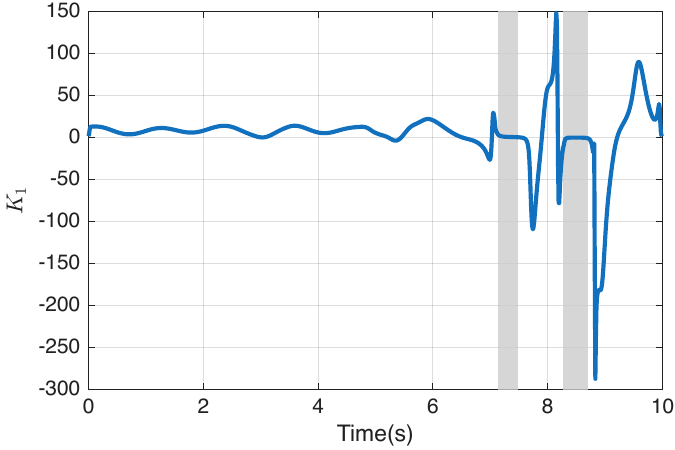}
    \caption{$K_{1}(t)$.}
    \label{fig:acrobot_feedback1}
    \end{subfigure}
    \hfill
    \begin{subfigure}{0.48\textwidth}
            \includegraphics[width=1\linewidth]{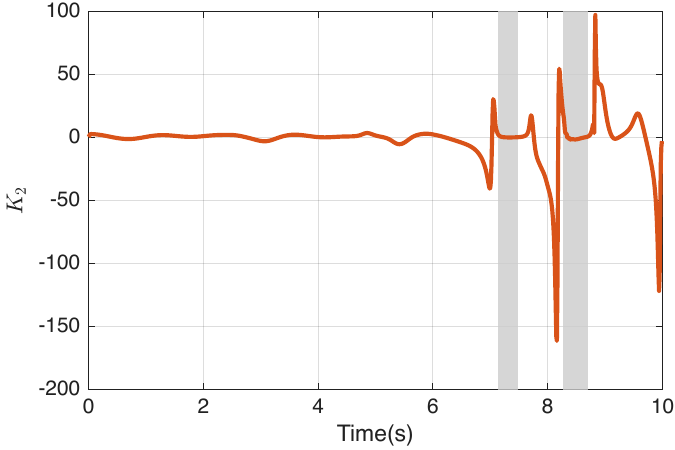}
    \caption{$K_{2}(t)$.}
    \label{fig:acrobot_feedback2}
    \end{subfigure}

    \begin{subfigure}{0.48\textwidth}
            \includegraphics[width=1\linewidth]{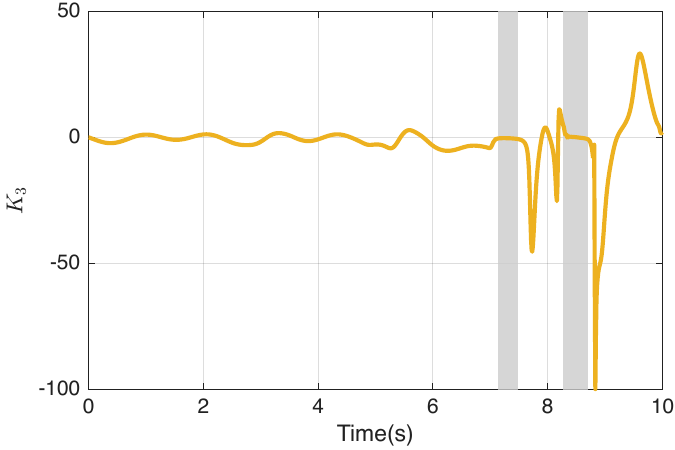}
    \caption{$K_{3}(t)$.}
    \label{fig:acrobot_feedback3}
    \end{subfigure}
    \hfill
    \begin{subfigure}{0.48\textwidth}
            \includegraphics[width=1\linewidth]{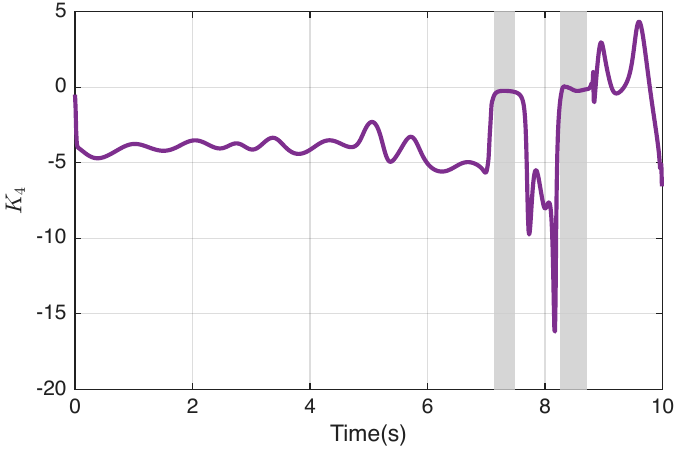}
    \caption{$K_{4}(t)$.}
    \label{fig:acrobot_feedback4}
    \end{subfigure}
    \caption{Plots of control input and components of feedback, $\mathbf{K}_t^{*}$, for the acrobot swing-up. The shaded region in the plots show the area where control nears saturation and the feedback terms tend to zero.}
    \label{fig:acrobot_feedback}
\end{figure}

\section{Conclusions and Future Work}
\label{sec:conclusion}

In this paper, we introduced Box-iLQR, an algorithm that integrates a logarithmic barrier Interior Point Method into the iLQR framework to solve trajectory optimization problems with box constraints on states and controls. By systematically reducing the barrier parameter, the method converges to the optimal constrained solution, a result we validated on several problems. A key benefit of this formulation is that the barrier terms act as a natural regularizer, ensuring the positive-definiteness and preserving convexity.

Future work will pursue two main directions. First, we will explore adaptive update strategies for the barrier parameter to accelerate convergence by avoiding the need to fully solve each intermediate subproblem. Second, we will analyze the properties of the resulting ``constraint-aware'' feedback law, which holds potential for more robust performance in model-predictive control applications. Overall, Box-iLQR represents a significant step toward a fast, reliable, and general iLQR solver for constrained optimal control.


%
%
%
\bibliographystyle{splncs04}
\bibliography{references.bib}

@INPROCEEDINGS{CL_DDP,
  author={Tassa, Yuval and Mansard, Nicolas and Todorov, Emo},
  booktitle={2014 IEEE International Conference on Robotics and Automation (ICRA)}, 
  title={Control-limited differential dynamic programming}, 
  year={2014},
  volume={},
  number={},
  pages={1168-1175},
  doi={10.1109/ICRA.2014.6907001}}

@INPROCEEDINGS{DDP_IPM,
  author={Prabhu, Siddharth and Rangarajan, Srinivas and Kothare, Mayuresh},
  booktitle={2025 American Control Conference (ACC)}, 
  title={Differential Dynamic Programming with Stagewise Equality and Inequality Constraints Using Interior Point Method}, 
  year={2025},
  volume={},
  number={},
  pages={2255-2261},
  doi={10.23919/ACC63710.2025.11108083}}

@inproceedings{Jackson2019ALiLQRT,
  title={AL-iLQR Tutorial},
  author={Brian E. Jackson},
  year={2019},
  url={https://api.semanticscholar.org/CorpusID:209455095}
}

@book{fiacco1968nonlinear,
  title     = {Nonlinear Programming: Sequential Unconstrained Minimization Techniques},
  author    = {Fiacco, Anthony V. and McCormick, Garth P.},
  year      = {1968},
  publisher = {Wiley},
  address   = {New York},
  note      = {Reprinted in 1990 by SIAM}
}

@article{karmarkar1984new,
  title     = {A new polynomial-time algorithm for linear programming},
  author    = {Karmarkar, Narendra},
  journal   = {Combinatorica},
  volume    = {4},
  number    = {4},
  pages     = {373--395},
  year      = {1984},
  publisher = {Springer}
}

@article{wright1992interior,
  title     = {Interior methods for constrained optimization},
  author    = {Wright, Stephen J.},
  journal   = {Acta Numerica},
  volume    = {1},
  pages     = {341--407},
  year      = {1992},
  publisher = {Cambridge University Press}
}

@book{wright1997primal,
  title     = {Primal-Dual Interior-Point Methods},
  author    = {Wright, Stephen J.},
  year      = {1997},
  publisher = {SIAM},
  address   = {Philadelphia, PA}
}

@book{nocedal2006numerical,
  title     = {Numerical Optimization},
  author    = {Nocedal, Jorge and Wright, Stephen J.},
  edition   = {Second},
  year      = {2006},
  publisher = {Springer},
  address   = {New York}
}

@article{betts1998survey,
  title     = {Survey of numerical methods for trajectory optimization},
  author    = {Betts, John T.},
  journal   = {Journal of Guidance, Control, and Dynamics},
  volume    = {21},
  number    = {2},
  pages     = {193--207},
  year      = {1998},
  publisher = {AIAA}
}

@book{betts2010practical,
  title     = {Practical Methods for Optimal Control and Estimation Using Nonlinear Programming},
  author    = {Betts, John T.},
  year      = {2010},
  series    = {Advances in Design and Control},
  volume    = {19},
  publisher = {SIAM},
  address   = {Philadelphia, PA}
}

@article{wachter2006implementation,
  title     = {On the implementation of a primal-dual interior point filter line search algorithm for large-scale nonlinear programming},
  author    = {W{\"a}chter, Andreas and Biegler, Lorenz T.},
  journal   = {Mathematical Programming},
  volume    = {106},
  number    = {1},
  pages     = {25--57},
  year      = {2006},
  publisher = {Springer}
}

@misc{parunandi2020d2c,
      title={D2C 2.0: Decoupled Data-Based Approach for Learning to Control Stochastic Nonlinear Systems via Model-Free ILQR}, 
      author={Karthikeya S Parunandi and Aayushman Sharma and Suman Chakravorty and Dileep Kalathil},
      year={2020},
      eprint={2002.07368},
      archivePrefix={arXiv},
      primaryClass={math.OC},
      url={https://arxiv.org/abs/2002.07368}, 
}

@article{rand2c,
    author = {Wang, Ran and Sharma, Aayushman and Parunandi, Karthikeya S. and Goyal, Raman and Mohamed, Mohamed Naveed Gul and Chakravorty, Suman},
    title = {The Search for Feedback in Reinforcement Learning},
    journal = {Journal of Dynamic Systems, Measurement, and Control},
    volume = {147},
    number = {6},
    pages = {061002},
    year = {2025},
    month = {06},
    issn = {0022-0434},
    doi = {10.1115/1.4068705},
    url = {https://doi.org/10.1115/1.4068705},
    eprint = {https://asmedigitalcollection.asme.org/dynamicsystems/article-pdf/147/6/061002/7501833/ds-24-1123.pdf},
}

@INPROCEEDINGS{ILQR_todorov,
  author={Tassa, Yuval and Erez, Tom and Todorov, Emanuel},
  booktitle={2012 IEEE/RSJ International Conference on Intelligent Robots and Systems}, 
  title={Synthesis and stabilization of complex behaviors through online trajectory optimization}, 
  year={2012},
  volume={},
  number={},
  pages={4906-4913},
  doi={10.1109/IROS.2012.6386025}}

@INPROCEEDINGS{abhi_convexity,
  author={Abhijeet and Mohamed, Mohamed Naveed Gul and Sharma, Aayushman and Chakravorty, Suman},
  booktitle={2024 IEEE 63rd Conference on Decision and Control (CDC)}, 
  title={Convexity in Optimal Control Problems}, 
  year={2024},
  volume={},
  number={},
  pages={261-266},
  doi={10.1109/CDC56724.2024.10886652}}

@article{lantoine2012hybrid,
  title     = {A Hybrid Differential Dynamic Programming Algorithm for Constrained Optimal Control Problems. Part 1: Theory},
  author    = {Lantoine, Gr{\'e}gory and Russell, Ryan P.},
  journal   = {Journal of Optimization Theory and Applications},
  volume    = {154},
  number    = {2},
  pages     = {382--417},
  year      = {2012},
  publisher = {Springer},
  doi       = {10.1007/s10957-012-0039-0},
  issn      = {1573-2878},
  url       = {https://doi.org/10.1007/s10957-012-0039-0}
}

@article{lantoine2012hybrid_part2,
  title     = {A Hybrid Differential Dynamic Programming Algorithm for Constrained Optimal Control Problems. Part 2: Application},
  author    = {Lantoine, Gr{\'e}gory and Russell, Ryan P.},
  journal   = {Journal of Optimization Theory and Applications},
  volume    = {154},
  number    = {2},
  pages     = {418--442},
  year      = {2012},
  publisher = {Springer},
  doi       = {10.1007/s10957-012-0038-1},
  issn      = {1573-2878},
  url       = {https://doi.org/10.1007/s10957-012-0038-1}
}

@misc{abhi_DDP_ILQR,
      title={A Sequential Quadratic Programming Perspective on Optimal Control}, 
      author={Abhijeet and Suman Chakravorty},
      year={2025},
      eprint={2510.03475},
      archivePrefix={arXiv},
      primaryClass={math.OC},
      url={https://arxiv.org/abs/2510.03475}, 
}

@INPROCEEDINGS{naveed_feedback,
  author={Mohamed, Mohamed Naveed Gul and Chakravorty, Suman and Goyal, Raman and Wang, Ran},
  booktitle={2022 American Control Conference (ACC)}, 
  title={On the Feedback Law in Stochastic Optimal Nonlinear Control}, 
  year={2022},
  volume={},
  number={},
  pages={970-975},
  doi={10.23919/ACC53348.2022.9867673}}

@book{bryson1969,
  title     = {Applied Optimal Control: {{Optimization}}, {{Estimation}}, and {{Control}}},
  author    = {Bryson, Jr., Arthur E. and Ho, Yu-Chi},
  year      = {1969},
  publisher = {Blaisdell Publishing Company},
  address   = {Waltham, MA}
}

@article{mayne1965ddp,
  title={A second-order gradient method for determining optimal trajectories of non-linear discrete-time systems},
  author={Mayne, David Q},
  journal={International Journal of Control},
  volume={3},
  number={1},
  pages={85--95},
  year={1965}
}

@article{mastalli2020fddp,
  title={Feasibility-driven differential dynamic programming},
  author={Mastalli, Carlos and Havoutis, Ioannis and Focchi, Michele and Caldwell, Darwin G and Semini, Claudio},
  journal={IEEE Transactions on Robotics},
  volume={36},
  number={5},
  pages={1364--1378},
  year={2020},
  publisher={IEEE}
}

@inproceedings{wang2022leg_based_sconmip,
  author    = {Wang, Anji and Wensing, Patrick M.},
  title     = {Leg-based {S-CON-MIP} for Rapid and Robust Robot Maneuvers},
  booktitle = {2022 International Conference on Robotics and Automation (ICRA)},
  pages     = {273--279},
  year      = {2022},
  organization = {IEEE}
}

@inproceedings{bhatia2021trajectory_slipping,
  author    = {Bhatia, Tirthak and Wang, Anji and Wensing, Patrick M.},
  title     = {Trajectory Optimization for Legged Robots with Slipping and Force-Controllable Contacts},
  booktitle = {2021 IEEE International Conference on Robotics and Automation (ICRA)},
  pages     = {12674--12680},
  year      = {2021},
  organization = {IEEE}
}

@inproceedings{ortiz2020viability_ddp,
  author    = {Ortiz, L. C. and Castro, G. and Wensing, P. M.},
  title     = {Viability and Performance of {DDP}-based Controllers for Dynamic Legged Robot Locomotion},
  booktitle = {2020 IEEE-RAS 20th International Conference on Humanoid Robots (Humanoids)},
  pages     = {186--193},
  year      = {2020},
  organization = {IEEE}
}

@article{zucker2013chomp,
  title={CHOMP: Gradient optimization techniques for efficient motion planning},
  author={Zucker, Matthew and Ratliff, Nathan and Dragan, Anca D and Pivtoraiko, Mykhaylo and Kuffner, James J and Dellin, Christopher M and Srinivasa, Siddhartha S},
  journal={The International Journal of Robotics Research},
  volume={32},
  number={9-10},
  pages={1145--1167},
  year={2013},
  publisher={SAGE Publications Sage UK: London, England}
}

@inproceedings{schulman2013finding,
  title={Finding locally optimal, collision-free trajectories with sequential convex optimization},
  author={Schulman, John and Ho, Jonathan and Lee, Alex and Awwal, Ibrahim and Bradlow, Henry and Abbeel, Pieter},
  booktitle={Robotics: Science and Systems},
  volume={9},
  year={2013},
  doi={10.15607/RSS.2013.IX.033}
}

@inproceedings{kalakrishnan2011storc,
  title={{STORC}: Stochastic trajectory optimization for motion planning},
  author={Kalakrishnan, Mrinal and Chitta, Sachin and Todorov, Evangelos and Righetti, Ludovic and Schaal, Stefan},
  booktitle={2011 IEEE International Conference on Robotics and Automation},
  pages={5551--5556},
  year={2011},
  organization={IEEE}
}

@article{Sharma2026Reduced,
  author    = {Sharma, A. and Chakravorty, S.},
  title     = {{A Reduced-Order Model-Based Reinforcement Learning Approach to the Control of Nonlinear Partial Differential Equations}},
  journal   = {Journal of Dynamic Systems, Measurement, and Control},
  volume    = {148},
  number    = {2},
  articleno = {021016},
  year      = {2026},
  month     = {mar},
  doi       = {10.1115/1.4070654},
  publisher = {ASME}
}

@INPROCEEDINGS{11107931,
  author={Sharma, Aayushman and Chakravorty, Suman},
  booktitle={2025 American Control Conference (ACC)}, 
  title={Data-Driven Modeling for Nonlinear Optimal Control}, 
  year={2025},
  volume={},
  number={},
  pages={4017-4022},
  doi={10.23919/ACC63710.2025.11107931}}

@INPROCEEDINGS{naveed_infinite,
  author={Gul Mohamed, Mohamed Naveed and Goyal, Raman and Chakravorty, Suman},
  booktitle={2023 62nd IEEE Conference on Decision and Control (CDC)}, 
  title={An Optimal Solution to Infinite Horizon Nonlinear Control Problems}, 
  year={2023},
  volume={},
  number={},
  pages={1643-1648},
  doi={10.1109/CDC49753.2023.10384307}}

@INPROCEEDINGS{raman_pod,
  author={Goyal, Raman and Gul Mohamed, Mohamed Naveed and Wang, Ran and Sharma, Aayushman and Chakravorty, Suman},
  booktitle={2025 American Control Conference (ACC)}, 
  title={Information-state based Approach to the Optimal Output Feedback Control of Nonlinear Systems}, 
  year={2025},
  volume={},
  number={},
  pages={4011-4016},
  doi={10.23919/ACC63710.2025.11107601}}

@ARTICLE{raman_tnnls,
  author={Goyal, Raman and Naveed Gul Mohamed, Mohamed and Wang, Ran and Sharma, Aayushman and Chakravorty, Suman},
  journal={IEEE Transactions on Neural Networks and Learning Systems}, 
  title={Information-State-Based Reinforcement Learning for the Control of Partially Observed Nonlinear Systems}, 
  year={2025},
  volume={36},
  number={12},
  pages={20386-20400},
  doi={10.1109/TNNLS.2025.3593259}}

@article{rao1998application,
  title     = {Application of Interior-Point Methods to Model Predictive Control},
  author    = {Rao, Christopher V. and Wright, Stephen J. and Rawlings, James B.},
  journal   = {Journal of Optimization Theory and Applications},
  volume    = {99},
  number    = {3},
  pages     = {723--757},
  year      = {1998},
  publisher = {Springer},
  doi       = {10.1023/A:1021711402723},
}

@inproceedings{howell2019altro,
  author    = {Howell, Taylor A. and Jackson, Brian E. and Manchester, Zachary},
  booktitle = {2019 IEEE/RSJ International Conference on Intelligent Robots and Systems (IROS)},
  title     = {{ALTRO}: A Fast Solver for Constrained Trajectory Optimization},
  year      = {2019},
  address   = {Macau, China},
  pages     = {7666--7673},
  doi       = {10.1109/IROS40897.2019.8967839}
}

@misc{vanroye2023fatropfastconstrained,
      title={FATROP : A Fast Constrained Optimal Control Problem Solver for Robot Trajectory Optimization and Control}, 
      author={Lander Vanroye and Ajay Sathya and Joris De Schutter and Wilm Decré},
      year={2023},
      eprint={2303.16746},
      archivePrefix={arXiv},
      primaryClass={math.OC},
      url={https://arxiv.org/abs/2303.16746}, 
}

@INPROCEEDINGS{FATROP,
  author={Vanroye, Lander and Sathya, Ajay and De Schutter, Joris and Decré, Wilm},
  booktitle={2023 IEEE/RSJ International Conference on Intelligent Robots and Systems (IROS)}, 
  title={FATROP: A Fast Constrained Optimal Control Problem Solver for Robot Trajectory Optimization and Control}, 
  year={2023},
  volume={},
  number={},
  pages={10036-10043},
  doi={10.1109/IROS55552.2023.10342336}}

@article{lin1991differential,
  author  = {Lin, T. C. and Arora, J. S.},
  title   = {{Differential dynamic programming technique for constrained optimal control}},
  journal = {Computational Mechanics},
  year    = {1991},
  volume  = {9},
  pages   = {41--53},
  doi     = {10.1007/BF00369914}
}

@article{Giftthaler2017APA,
  title={A projection approach to equality constrained iterative linear quadratic optimal control},
  author={Markus Giftthaler and Jonas Buchli},
  journal={2017 IEEE-RAS 17th International Conference on Humanoid Robotics (Humanoids)},
  year={2017},
  pages={61-66},
  url={https://api.semanticscholar.org/CorpusID:8565564}
}

@article{Xie2017DifferentialDP,
  title={Differential dynamic programming with nonlinear constraints},
  author={Zhaoming Xie and C. Karen Liu and Kris K. Hauser},
  journal={2017 IEEE International Conference on Robotics and Automation (ICRA)},
  year={2017},
  pages={695-702},
  url={https://api.semanticscholar.org/CorpusID:7648776}
}

@INPROCEEDINGS{emo_ILQG,
  author={Todorov, E. and Weiwei Li},
  booktitle={Proceedings of the 2005, American Control Conference, 2005.}, 
  title={A generalized iterative LQG method for locally-optimal feedback control of constrained nonlinear stochastic systems}, 
  year={2005},
  volume={},
  number={},
  pages={300-306 vol. 1},
  doi={10.1109/ACC.2005.1469949}}

@article{waltz2006interior,
  author  = {Waltz, R. A. and Morales, J. L. and Nocedal, J. and Orban, D.},
  title   = {{An interior algorithm for nonlinear optimization that combines line search and trust region steps}},
  journal = {Mathematical Programming},
  year    = {2006},
  volume  = {107},
  number  = {3},
  pages   = {391--408},
  doi     = {10.1007/s10107-004-0560-5}
}

\end{document}